\documentclass[a4paper,final]{article}
\usepackage[british]{babel}
\usepackage{amsmath, amssymb, amsthm, amsfonts}
\usepackage{ifdraft}
\usepackage{microtype}
\usepackage{url}
\DeclareMathAlphabet{\mathpzc}{OT1}{pzc}{m}{it}
\usepackage[utf8]{inputenc}
\usepackage{aliascnt}
\usepackage[colorlinks=true,linkcolor=blue,draft=false]{hyperref}
\usepackage{tikz}
\usetikzlibrary{arrows,automata}
\usepackage{mathtools}
\usepackage{enumerate}
\usepackage[all]{xy}
\usepackage{paralist}
\usepackage{extarrows}

\usepackage[notref,notcite,color]{showkeys}

\ifdraft{
\newcommand{\colorS}[1]{{\color{blue} #1}}
\newcommand{\colorV}[1]{{\color{red} #1}}
\newcommand{\itemS}[1]{\colorS{\item[S:] #1}}
\newcommand{\itemV}[1]{\colorV{\item[V:] #1}}
\newcommand{\commS}[1]{{\marginpar{\tiny \colorS{#1}}}}
\newcommand{\commV}[1]{{\marginpar{\tiny \colorV{#1}}}}
\newcommand{\comment}[1]{{\color{gray} #1}}
}
{
\newcommand{\colorS}[1]{#1}
\newcommand{\colorV}[1]{#1}
\newcommand{\itemS}[1]{}
\newcommand{\itemV}[1]{}
\newcommand{\commS}[1]{}
\newcommand{\commV}[1]{}
\newcommand{\comment}[1]{}
}

\makeatletter
\def\myfooter{\xdef\@thefnmark{}\@footnotetext}
\makeatother

\title{On the penetration distance in Garside monoids}
\author{\colorV{Volker Gebhardt}, \colorS{Stephen Tawn}}
\date{16th January 2016}

\hypersetup{pdftitle={On the penetration distance in Garside monoids},
            pdfauthor={Volker Gebhardt, Stephen Tawn},
            pdfkeywords={},
            pdfsubject={},
            pdfdisplaydoctitle={true}}

\newcommand{\NN}{\mathbb{N}}
\newcommand{\ZZ}{\mathbb{Z}}
\newcommand{\RR}{\mathbb{R}}

\newcommand{\Artin}[1]{\mathsf{#1}}

\newcommand{\st}{\boldsymbol{\mid}}

\newcommand{\Union}{\bigcup}
\newcommand{\intersect}{\cap}

\let\emptyset\varnothing

\def\setminus{\mathbin{\raisebox{0.25ex}{$\smallsetminus$}}}
\let\le\leqslant

\let\ge\geqslant

\newcommand{\from}{\colon\!}
\newcommand{\join}{\vee}
\renewcommand{\Join}{\bigvee}
\newcommand{\meet}{\wedge}

\newcommand{\rightJoin}{\mathop{\widetilde{\Join}}}
\let\epsilon\varepsilon
\newcommand{\prefix}{\preccurlyeq}
\newcommand{\suffix}{\succcurlyeq}
\newcommand{\under}{\backslash}
\DeclareMathOperator{\rev}{rev}
\newcommand{\ldot}{\mathbin{.}}

\DeclareMathOperator{\Div}{Div}
\newcommand{\Ess}{\mathpzc{Ess}}
\newcommand{\Rig}{\mathpzc{Rig}}

\newcommand{\pd}{\mathrm{pd}}
\newcommand{\Atoms}{\mathcal{A}}
\newcommand{\Simples}{\mathcal{D}}
\newcommand{\pSimples}{\Simples^{\!{}^{\circ}\!}}
\newcommand{\Acceptor}{\Gamma}
\newcommand{\PSeqAcceptor}{\Pi}
\newcommand{\cl}{\mathrm{cl}}
\renewcommand{\complement}{\partial}
\newcommand{\rightcomplement}{\widetilde{\complement}}
\newcommand{\id}{\mathbf{1}}
\newcommand{\PSeq}{\textsc{PSeq}}

\newcommand{\Start}{S}
\newcommand{\Fin}{F}
\newcommand{\NF}{\mathrm{NF}}
\newcommand{\ZS}{Zappa--Sz{\'e}p}

\let\zs=\bowtie

\newcommand{\Expect}{\mathbb{E}}
\newcommand{\calL}{\mathcal{L}}

\newcommand{\calLbar}{\overline{\mathcal{L}}}

\theoremstyle{plain}
\newtheorem{theorem}{Theorem}

\newaliascnt{lemma}{theorem}
\newtheorem{lemma}[lemma]{Lemma}
\aliascntresetthe{lemma}

\newaliascnt{proposition}{theorem}
\newtheorem{proposition}[proposition]{Proposition}
\aliascntresetthe{proposition}

\newaliascnt{corollary}{theorem}
\newtheorem{corollary}[corollary]{Corollary}
\aliascntresetthe{corollary}

\newaliascnt{conjecture}{theorem}

\aliascntresetthe{conjecture}

\theoremstyle{remark}

\newaliascnt{claim}{theorem}

\aliascntresetthe{claim}

\newtheorem*{claim*}{Claim}
\newtheorem*{remark}{Remark}

\theoremstyle{definition}

\newaliascnt{definition}{theorem}
\newtheorem{definition}[definition]{Definition}
\aliascntresetthe{definition}

\newaliascnt{example}{theorem}
\newtheorem{example}[example]{Example}
\aliascntresetthe{example}

\newaliascnt{notation}{theorem}
\newtheorem{notation}[notation]{Notation}
\aliascntresetthe{notation}

\addto\extrasbritish{

}

\makeatletter
\def\MR@url=#1 #2={http://www.ams.org/mathscinet-getitem?mr=#1}

\makeatother

\begin{document}

\maketitle
\myfooter{Both authors acknowledge support under Australian Research Council's Discovery Projects funding scheme (project number DP1094072). Volker Gebhardt acknowledges support under the Government of Spain Projects MTM2010-19355 and MTM2013-44233-P.}
\myfooter{MSC-class: 20F36, 43A07 (Primary) 20F69, 60B15 (Secondary)}
\myfooter{Keywords: Garside monoids; Artin monoids; random elements; normal forms; regular language; penetration distance; essential transitivity}

\begin{abstract}
We prove that the exponential growth rate of the regular language of
penetration sequences is smaller than the growth rate of the regular language
of normal form words, if the acceptor of the regular language of normal form
words is strongly connected.
Moreover, we show that the latter property is satisfied for all irreducible
Artin monoids of spherical type, extending a result by Caruso.

   Our results establish that the expected value of the penetration distance
$pd(x,y)$ in an irreducible Artin monoid of spherical type is bounded
independently of the length of $x$, if $x$ is chosen uniformly among all
elements of given canonical length and $y$ is chosen uniformly among all
atoms; the latter in particular explains observations made by Thurston
in the context of the braid group, and it shows that all irreducible Artin
monoids of spherical type exhibit an analogous behaviour.
Our results also give an affirmative answer to a question posed by Dehornoy.
\end{abstract}

\section{Introduction}

Random walks on discrete infinite groups are a subject that has received substantial interest; see for instance~\cite{KaimanovichVershik,Woess2000} and the references therein.

Random walks on the 3-strand braid group were analysed in~\cite{MairesseMatheus}.
A complete understanding of random walks on braid groups on a larger number of strands, let alone more general classes of groups such as Artin--Tits groups or Garside groups, has not been achieved yet.%
\medskip

A topic closely related to random walks on Garside groups is the behaviour of the Garside normal forms of random elements:
Fixing a position in the normal form, one obtains an induced distribution on the set of simple elements (that is, on the symmetric group in the case of braids) which can be studied.
In~\cite{GT13}, we investigated these induced distributions and observed convergence.
More precisely, experimental data suggests that, except for an initial and a final region whose lengths are uniformly bounded, the distributions of the factors of the normal form of sufficiently long random braids depend neither on the position in the normal form nor on the lengths of the random braids.%
\medskip

Stabilisation phenomena in the normal forms of random braids were first observed by Thurston in~\cite[Chapter~9]{braid_epsteinetal}:
A remark at the end of the chapter mentions that in experiments with a small number of strands, the multiplication of a braid~$x$ in normal form by a simple braid appeared unlikely to affect all factors of the normal form of~$x$; moreover, if one ignores factors that are modified only by conjugation by the Garside element (the half-twist in the case of the braid group), the number of factors that are modified appeared to be ``roughly constant''.
Based on these observations, an alternative algorithm for computing the normal form of a braid was suggested, but the question of proving its efficiency was left open.%
\medskip

In~\cite{GT13}, we formalised analysing the number of factors in the normal form of an element~$x$ of a Garside monoid~$M$ that are modified non-trivially (that is, other than by conjugation by the Garside element) when computing the normal form of the product~$x\cdot y$; we called this number the \emph{penetration distance} $\pd(x,y)$:
%
%
For specific distributions of~$x$ and~$y$, we described the behaviour of the expected value $\Expect[\pd]$ of the penetration distance in terms of the growth rates~$\alpha$ and~$\beta$ of two regular languages (cf.\ \autoref{S:BackgroundPenetration}).
More precisely, we proved that, if $x$ is chosen uniformly among all elements of canonical length~$k$ and~$y$ is chosen uniformly among the atoms, then $\Expect[\pd]$ is bounded above independently of the value of~$k$ if $\alpha<\beta$ holds~\cite[Theorem~4.7]{GT13}.
By explicit calculations, we showed that the latter condition is satisfied for some irreducible Artin monoids of spherical type with a small number of atoms.
\bigskip

In the present paper, we relate the condition $\alpha<\beta$ to certain structural properties of the lattice of simple elements (more precisely, to the connectedness of the acceptor of the regular language of normal forms), and we prove that these properties are satisfied for all irreducible Artin monoids of spherical type.

As a consequence, we establish here that the expected value of the penetration distance is bounded in the sense of~\cite[Theorem~4.7]{GT13} for all irreducible Artin monoids of spherical type.
Thus, the results obtained in this paper in particular explain the behaviour that was observed by Thurston in~\cite[Chapter~9]{braid_epsteinetal} in the case of the braid monoid, and they show that all irreducible Artin monoids of spherical type exhibit an analogous behaviour.
Moreover, our results imply that in all irreducible Artin monoids of spherical type, the number of normal form sequences of length~$k$ grows as fast as the number of normal form sequences of length~$k$ that start with a given proper simple element, answering in the affirmative a question posed by Dehornoy in the context of the braid monoid~\cite{DehornoyJCTA07}.

We also prove that the expected value of the penetration distance is unbounded for \ZS{} products of 
irreducible Artin monoids of spherical type.%
\bigskip

The structure of the paper is as follows.
In \autoref{S:BackgroundGarside}, we recall the basic concepts of normal forms in Garside groups.
In \autoref{S:BackgroundArtin}, we briefly recall Artin monoids of spherical type; experts may skip this section.
In \autoref{S:BackgroundPenetration}, we recall the regular languages defined in~\cite{GT13} to study the expected value of the penetration distance.
\autoref{S:BackgroundZappaSzep} recalls the notion of \ZS{} products.
In \autoref{S:Essential} we define the notions of \emph{essential} simple elements and \emph{essential transitivity}, which will be used in \autoref{S:GrowthRates} to compare the growth rates~$\alpha$ and~$\beta$.
\autoref{S:ArtinMonoids} establishes that irreducible Artin monoids of spherical type are essentially transitive and have the property that all proper simple elements are essential, and thus that the results of the preceding sections can be applied to this class of monoids.
\medskip

We thank the Institute for Mathematics at the University of Seville (IMUS) for providing access to a 64-node 512\,GB RAM computer; this equipment was used for some of the computations mentioned in \autoref{S:ArtinMonoids}.

\section{Background}\label{S:Background}

In order to fix notation, we briefly recall the main concepts used in the paper.

\subsection{Garside monoids and Garside normal form}\label{S:BackgroundGarside}

We refer to \cite{DehornoyParis99,Dehornoy02,GarsideBook} for details.

Let $M$ be a monoid with identity element~$\id$.
The monoid $M$ is called \emph{left-cancellative} if for any $x,y,y'$ in $M$, the equality $xy = xy'$ implies $y = y'$.
Similarly, $M$ is called \emph{right-cancellative} if for any $x,y,y'$ in $M$, the equality $yx = y'x$ implies $y = y'$.
The monoid~$M$ is called \emph{cancellative} if it is both left-cancellative and right-cancellative.

For $x,y\in M$, we say that~$x$ is a \emph{left-divisor} or \emph{prefix} of~$y$, writing $x \prefix_M y$, if there exists an element $u\in M$ such that $y = xu$.  If the monoid is obvious, we simply write $x \prefix y$ to reduce clutter.
Similarly, we say that~$x$ is a \emph{right-divisor} or \emph{suffix} of~$y$,
writing $y\suffix_M x$ or $y\suffix x$, if there exists $u\in M$ such that $y = ux$.
If $M$ does not contain any non-trivial invertible elements, then the relation $\prefix$ is a partial order if $M$ is left-cancellative, and the relation $\suffix$ is a partial order if $M$ is right-cancellative.

An element $a\in M\setminus\{\id\}$ is called an \emph{atom} if whenever $a = uv$ for $u,v\in M$, either $u = \id$ or $v = \id$ holds.  The existence of atoms implies that $M$ does not contain any non-trivial invertible elements.
The monoid $M$ is said to be
\emph{atomic} if it is generated by its set $\Atoms$ of atoms and if for every
element $x\in M$ there is an upper bound on the length of
decompositions of $x$ as a product of atoms, that is, if
$||x||_{\Atoms} := \sup\{ k\in\NN : x=a_1\cdots a_k \text{ with }a_1,\ldots,a_k\in\Atoms \} < \infty$.

An element $d\in M$ is called \emph{balanced}, if the set of its left-divisors is equal to the set of its right-divisors.  In this case, we write $\Div(d)$ for the set of (left- and right-) divisors of $d$.

\begin{definition}
  A \emph{quasi-Garside structure} is a pair $(M, \Delta)$ where $M$
  is a monoid and $\Delta$ is an element of $M$ such that
  \begin{enumerate}[ (a)] \itemsep 0em \vspace{-0.25\topskip}
  \item $M$ is cancellative and atomic,
  \item the prefix and suffix relations are lattice orders, that is, for any pair of elements there exist a unique least common upper bound and a unique greatest common lower bound with respect to $\prefix$ respectively $\suffix$,
  \item $\Delta$ is balanced and $M$ is generated by $\Div(\Delta)$.
  \end{enumerate}
  The element $\Delta$ is called a \emph{quasi-Garside element} for~$M$.
  
  If $\Div(\Delta)$ is finite then we say that $(M, \Delta)$ is a \emph{Garside structure} and call~$\Delta$ a \emph{Garside element} for~$M$.
\end{definition}

\begin{remark}
  If $(M,\Delta)$ is a quasi-Garside structure in the above sense, then in the terminology of \cite{GarsideBook}, the set $\Div(\Delta)$ forms a bounded Garside family for the monoid $M$.
\end{remark}

\begin{definition}
  A \emph{(quasi-)Garside monoid} is a cancellative monoid~$M$ together with a fixed (quasi-)Garside structure
  $(M, \Delta)$.
  The elements of $\Div(\Delta)$ are called the \emph{simple elements} of the (quasi-)Garside monoid~$M$.
\end{definition}

\begin{remark}
  If a cancellative monoid admits one (quasi-)Garside structure, it admits infinitely many; cf.\ \autoref{D:framing}.
  Indeed, for the monoid $\NN^n$, any element $(x_1,\ldots,x_n)\in (\NN_{>0})^n$ can be chosen as the Garside element.
  
  The notion of simple elements, and thus the normal form of an element, depends on the specific (quasi-)Garside element, so different (quasi-)Garside structures for the same monoid must be distinguished.
\end{remark}

\begin{notation}
If $M$ is a left-cancellative atomic monoid, then least common upper bounds and greatest common lower bounds are unique if they exist.  In this situation, we will write $x \join y$ for the $\prefix$-least common upper bound of $x,y\in M$ if it exists, and we write $x \meet y$ for their $\prefix$-greatest common lower bound if it exists.
If $x,y\in M$ admit a $\prefix$-least common upper bound, we define $x\under y$ as the unique element of $M$ satisfying $x(x\under y)=x\join y$.
%

If $M$ is a (quasi-)Garside monoid with (quasi-)Garside element $\Delta$, we write~$\Atoms_M$ for the set of atoms,~$\Simples_M$ for the set of simple elements~$\Div(\Delta)$, and we define the set of \emph{proper simple elements} as $\pSimples_M = \Simples_M \setminus \{ \id, \Delta \}$.  To avoid clutter, we will usually drop the subscript if there is no danger of confusion.

For any $x\in\Simples$, there are unique elements $\complement x=\complement_M x\in\Simples$ and $\rightcomplement x = \rightcomplement_M x\in \Simples$ such that $x \complement x=\Delta = (\rightcomplement x) x$.
We define inductively $\complement^{k+1} x = \complement(\complement^k x)$ and $\rightcomplement^{k+1} x = \rightcomplement(\rightcomplement^k x)$ for $k\in\NN$.
As $\complement (\rightcomplement x) = x = \rightcomplement(\complement x)$ for any $x\in\Simples$, we can define $\complement^k x = \rightcomplement^{-k} x$ for any $k\in\ZZ$.
Clearly, $\complement^k x\in\pSimples$ iff $x\in\pSimples$.
Moreover, for any $x,y\in\Simples$, one has $x\prefix y$ iff $\complement x \suffix \complement y$ iff $\complement^2 x \prefix \complement^2 y$.

For $x \in M$ the \emph{starting set}~$\Start(x)$ of~$x$ is the set of atoms that are prefixes of~$x$.  Similarly, the \emph{finishing set}~$\Fin(x)$~of~$x$ is the set of atomic suffixes of~$x$:
\[
  \Start(x) = \{ a \in \Atoms : a \prefix x \} \quad\quad\quad
  \Fin(x) = \{ a \in \Atoms : x \suffix a \}
\]

Given a set $X$ we will write $X^* = \Union_{i = 0}^{\infty} X^i$ for
the set of strings (of finite length) of elements of $X$.  We will write $\epsilon$ for
the empty string and separate the letters of a string with dots, for
example we will write $a \ldot b \ldot a \in \{a,b\}^*$.

Given a (quasi)-Garside monoid~$M$ with (quasi-)Garside element~$\Delta$, we can define the \emph{left
normal form} of an element by repeatedly extracting the
$\prefix$-greatest common lower bound of the element and $\Delta$.  More precisely, the
normal form of $x \in M$ is the unique word $\NF(x) = x_1 \ldot x_2
\ldot \cdots \ldot x_\ell$ in $(\Simples\setminus\{\id\})^*$ such that $x = x_1 x_2 \cdots
x_\ell$ and $x_i = \Delta \meet x_i x_{i+1} \cdots x_\ell$ for $i=1,\ldots,\ell$, or equivalently, $\complement x_{i-1} \meet x_i=\id$ for $i=2,\ldots,\ell$.
We write $x_1|x_2|\cdots|x_\ell$ for the word $x_1 \ldot x_2 \ldot \cdots \ldot x_\ell$
together with the proposition that this word is in normal form.

If $x_1|x_2|\cdots|x_\ell$ is the normal form of $x\in M$, we define the \emph{infimum} of $x$ as
$\inf(x) = \max\{ i\in\{1,\ldots,\ell\} : x_i = \Delta \}$, the \emph{supremum} of $x$ as $\sup(x) = \ell$, and the \emph{canonical length} of $x$ as $\cl(x) = \sup(x)-\inf(x)$.  Note that $\inf(x)$ is the largest integer~$i$ such that $\Delta^i\prefix x$ holds, and $\sup(x)$ is the smallest integer~$i$ such that $x\prefix \Delta^i$ holds.

The operation $\complement$ can be extended to all of $M$ by defining $\complement x$ to be the unique element such that $x \complement x = \Delta^{\sup(x)}$.
If $\NF(x) = x_1|x_2 \cdots|x_\ell$ is the normal form of~$x$ and $\inf(x)=k$, then
$\NF(\complement x) = \complement x_\ell| \complement^3 x_{\ell-1}| \cdots |\complement^{2(\ell-k)-1} x_{k+1}$.

Let $\calL=\calL_M$ be the language on the set $\pSimples$ of proper simple elements
consisting of all words (of finite length) in normal form, and write $\calL^{(n)}=\calL_M^{(n)}$ for
the subset consisting of words of length $n$:
\[
  \calL := \bigcup_{n\in\NN} \calL^{(n)} \quad\text{where}\quad
  \calL^{(n)} := \big\{ 
      x_1 \ldot \cdots \ldot x_n \in \big(\pSimples\big)^* :
      \forall i,\ \complement x_i \meet x_{i+1} = \id 
    \big\}
\]

We also define $\calLbar := \calLbar_M := \bigcup_{n\in\NN} \calLbar^{(n)}$, where
\[
  \calLbar^{(n)} := \calLbar_M^{(n)} := \big\{ 
      x_1 \ldot \cdots \ldot x_n \in (\Simples\setminus\{\id\})^* :
      \forall i,\ \complement x_i \meet x_{i+1} = \id 
    \big\} \;.
\]
\end{notation}

\begin{definition}\label{D:framing}
Given a (quasi-)Garside monoid~$M$ with (quasi-)Garside element~$\Delta$ and an integer~$k\ge1$, let~$M(k)$
denote the same monoid but with the (quasi-)Garside structure given by the
(quasi-)Garside element~$\Delta^k$.
We call~$M(k)$ the \emph{$k$-framing} of~$M$.
\end{definition}

As the partial orders $\prefix$ and $\suffix$ on $M$ and $M(k)$ are identical, so are the lattice operations $\join$, $\meet$ and $\under$.  It is obvious from the definitions that one has
$\complement_{M(k)} x = \complement_M x \Delta^m$, where $m = k\big\lceil\frac{\sup(x)}{k}\big\rceil-\sup(x)$.

\begin{lemma} \label{lemma:non-minimal-normal-form}
  If~$M$ is a Garside monoid and~$w$ is a word in $M(k)$-normal form, then replacing each letter of~$w$
  with the word for its $M$-normal form yields a word in $M$-normal form.
\end{lemma}
\begin{proof}
  Suppose that we have $M(k)$-simple elements $x, y \in \Simples_{M(k)}$ for which $x|y$ holds in $M(k)$, and whose $M$-normal forms are as follows:
  \begin{align*}
    \NF_M(x) &= x_1 |x_2 |\cdots |x_k \\
    \NF_M(y) &= y_1 |y_2 |\cdots |y_k
  \end{align*}
  It is sufficient to prove that $x_k | y_1$ in the original Garside structure of $M$.

  Let $m = \complement_M x_k \meet y_1$.  We have that $m \prefix y_1
  \prefix y$ and $m \prefix \complement_M x_k \prefix
  \complement_{M(k)} x$.  Hence $m \prefix \complement_{M(k)} x \meet
  y = \id$ and so $\complement_M x_k \meet y_1 = \id$ as required.
\end{proof}

\begin{remark}
The map from \autoref{lemma:non-minimal-normal-form} does not take $\calL_{M(k)}$ to $\calL_M$, as a proper
simple element of $M(k)$ may have non-zero infimum in $M$.
\end{remark}

\begin{lemma} \label{lemma:normal-form-non-minimal}
Given a Garside monoid~$M$ and a word $x_0 | x_1 |\cdots| x_\ell$ in $M$-normal form, let $x_j=\id$ for $j>\ell$ and let $y_i = x_{ik} x_{ik + 1} \cdots x_{(i+1)k - 1}$ for $i=0,\ldots,\left\lfloor\frac{\ell}{k}\right\rfloor$.

Then $y_0 \ldot y_1 \ldot\cdots \ldot y_{\left\lfloor\frac{\ell}{k}\right\rfloor}$ is in $M(k)$-normal form.
\end{lemma}
\begin{proof}
One has $\Delta \meet \complement_{M(k)} y_{i-1} \meet y_i = \complement_M x_{ik-1} \meet x_{ik} = \id$ for $i=1,\ldots,\left\lfloor\frac{\ell}{k}\right\rfloor$ and $x_{ik}\neq\id$, hence $y_i\neq\id$, for $i=0,\ldots,\left\lfloor\frac{\ell}{k}\right\rfloor$.
\end{proof}

\begin{definition}\label{D:ParabolicSubmonoid}
Let $M$ be a Garside monoid with set of atoms $\Atoms$,
let $\delta$ be a balanced simple element of $M$, and let $M_\delta$ be the submonoid of~$M$ generated by
$\{ a\in \Atoms : a \prefix\delta\}$.

$M_\delta$ is a \emph{parabolic submonoid} of $M$, if
$\{ x \in M : x\prefix \delta \} = \Simples \cap M_\delta$ holds.
\end{definition}

\comment{\begin{itemize}
  \itemS{This isn't necessarily a question for here: How does the
    set of parabolic submonoids of $M$ relate to those of $M(k)$?}
\end{itemize}}

\begin{proposition}[{\cite[Lemma 2.1]{Godelle07}}]\label{P:ParabolicSubmonoid}
If $M$ is a Garside monoid and~$\delta$ is a balanced simple element of~$M$ such that~$M_\delta$ is a parabolic submonoid of~$M$, then~$M_\delta$ is a sublattice of~$M$ for both~$\prefix$ and~$\suffix$ that is closed under the operation~$\under$ and its dual~$/$ defined using~$\suffix$.
In particular, $M_\delta$ is a Garside monoid with Garside element $\delta$.
\end{proposition}

\begin{remark}
If $M_\delta$ is a parabolic submonoid of~$M$, then for any $x\in M_\delta$, the left normal form of $x$ in the Garside monoid $M_\delta$ coincides with the left normal form of $x$ in the Garside monoid $M$.
\end{remark}

\subsection{Artin monoids}\label{S:BackgroundArtin}

For details we refer to~\cite{BrieskornSaito,Deligne}.

Let~$I$ be a finite set and let $C=(m_{i,j})_{i,j\in I}$ be a \emph{Coxeter matrix}, that is, a symmetric matrix with entries $m_{i,j} = m_{j,i} \in \{2,3,\ldots,\infty\}$ for $i\ne j\in I$ and $m_{i,i}=2$ for
$i\in I$.
Associated to the pair $(I,C)$ is the monoid~$A_{(I,C)}$ generated by the elements of~$I$ subject to the relations
\begin{equation}\label{braid-relations}
    \langle i,j\rangle^{m_{i,j}} = \langle j,i\rangle^{m_{j,i}}
    \,\text{ for $i\ne j\in I$ with $m_{i,j}<\infty$ ,}
\end{equation}
where
$
    \langle i,j\rangle^{m_{i,j}} := \underbrace{i\cdot j \cdot i\cdot j\cdots}_{m_{i,j} \textrm{factors}}
    \;.
$

As these relations are homogenous, all expressions of an element~$x\in A_{(I,C)}$ as words in the generators~$I$ have the same length, so the length~$|x|_I$ of~$x$ with respect to the generating set~$I$ is well defined.

It is convenient to represent the data~$(I,C)$ in the form of a labelled (undirected) graph with vertex set~$I$: two vertices~$i$ and~$j$ are joined by an edge labelled~$m_{i,j}$ iff $m_{i,j}>2$, where the label~$3$ usually is not written explicitly.
\medskip



The \emph{Coxeter group} $W_{(I,C)}$ associated to the pair $(I,C)$ is
the group given by the quotient of the monoid $A_{(I,C)}$ by the
relations $i^2 = \id$ for all $i \in I$.  This gives a natural
epimorphism $\pi_{(I,C)}$ from the monoid~$A_{(I,C)}$ onto
$W_{(I,C)}$.

A word $i_1\ldot i_2 \ldot\cdots\ldot i_k \in I^*$ is called \emph{reduced}, if there is no word of length less than~$k$ that represents the element $i_1\cdots i_k$ of $W_{(I,C)}$.

\begin{lemma}[{\cite[Theorem 3.3.1]{CombCox}}]
  Two reduced words represent the same element of $W_{(I,C)}$ if and
  only if they can be related by \eqref{braid-relations}.  In other
  words, two minimum length words for the same element of $W_{(I,C)}$
  must also represent the same element of $A_{(I,C)}$.
\end{lemma}

As a consequence of the previous lemma, we have a well-defined map
$r_{(I,C)}$ from $W_{(I,C)}$ to $A_{(I,C)}$ which, for each element of
$W_{(I,C)}$, takes a reduced word representing the element and
reinterprets it as an element of $A_{(I,C)}$.  This gives a
set-theoretic section of $\pi_{(I,C)}$, that is, one has
$\pi_{(I,C)} \circ r_{(I,C)} = \mathrm{id}_{W_{(I,C)}}$.
\medskip

The monoid~$A_{(I,C)}$ is said to be \emph{of spherical type}, if the corresponding Coxeter group~$W_{(I,C)}$ is finite.
In this case,~$W_{(I,C)}$ has a unique longest element and lifting this element, via~$r_{(I,C)}$, gives a Garside element in~$A_{(I,C)}$.
The Garside structure this induces is as follows: The atoms of~$A_{(I,C)}$ are the elements of~$I$.  The set of simple elements is the image of~$r_{(I,C)}$. (That is, an element of~$A_{(I,C)}$ is simple if and only if it cannot be written as a word containing the square of a generator.)
Moreover, the prefix and suffix orders on the set of simple elements correspond to the right and, respectively, left weak orders on the Coxeter group~$W_{(I,C)}$.
Throughout this paper we will refer to the monoid $A_{(I,C)}$ with this particular Garside structure as the \emph{Artin monoid of type $(I,C)$}.

As a consequence of the simple elements being precisely those elements that do not involve the square of a generator, the Garside normal form condition can be expressed in terms of starting and finishing sets:

\begin{theorem}[{\cite[Lemma 4.2]{Charney1993}}]\label{Artin-start-finish}
If $M$ is an Artin monoid of spherical type with set of atoms $\Atoms$ and $s\in\Simples$, then $\Start(\complement s) = \Atoms\setminus\Fin(s)$.
\end{theorem}

\begin{corollary}\label{Artin-normal-form}
Suppose that $M$ is an Artin monoid of spherical type.  Then, for $x, y \in \Simples$, one has $x | y$ if and only if $\Fin(x) \supseteq \Start(y)$ holds.
\end{corollary}
\begin{proof}
We have $x|y$ if and only if $\complement x \meet y = \id$.  The latter is equivalent to $\Start(\partial x)\intersect \Start(y)=\emptyset$, so the claim follows with \autoref{Artin-start-finish}.
\end{proof}

If $A_{(I,C)}$ is an Artin monoid and $J\subseteq I$, then the submonoid of~$A_{(I,C)}$ generated by~$J$ is a parabolic submonoid of~$A_{(I,C)}$.  Indeed, it is (isomorphic to) the Artin monoid~$A_{(J,C')}$, where~$(J,C')$ is the subgraph of the graph~$(I,C)$ spanned by the vertex set~$J$.
\medskip

If the graph~$(I,C)$ is connected, the Artin monoid~$A_{(I,C)}$ is called \emph{irreducible}.  As generators in different connected components of the graph $(I,C)$ commute, it is obvious that every Artin monoid can be decomposed (in a unique way) as the direct product of irreducible Artin monoids.

The classification of finite Coxeter groups yields a classification of the irreducible Artin monoids of spherical type:
\begin{theorem}[\cite{Coxeter35,BrieskornSaito,Deligne}]\label{T:ArtinClassification}
The irreducible Artin monoids of spherical type are precisely those in the following list:
\[\begin{array}{l@{\qquad}l}
\Artin{A}_n:\quad
    \begin{xy}
      0;<2em,0em>:<0em,2em>::
      (1,0)*+{1}="1";
      (2,0)*+{2}="2";
      (3,0)*+{3}="3";
      (3.75,0)="4";
      (4.7,0)="5";
      (5.5,0)*+{n}="6";
      {\ar@{-}     "1";"2"};
      {\ar@{-}     "2";"3"};
      {\ar@{-}     "3";"4"};
      {\ar@{..}    "4";"5"};
      {\ar@{-}     "5";"6"};
    \end{xy}
    & (n\ge 1)
  \\[1.5ex]
\Artin{B}_n:\quad
    \begin{xy}
      0;<2em,0em>:<0em,2em>::
      (0,0)*+{0}="0";
      (1,0)*+{1}="1";
      (2,0)*+{2}="2";
      (3,0)*+{3}="3";
      (3.75,0)="4";
      (4.6,0)="5";
      (6,0)*+{(n-1)}="6";
      {\ar@{-}^{4}  "0";"1"};
      {\ar@{-}     "1";"2"};
      {\ar@{-}     "2";"3"};
      {\ar@{-}     "3";"4"};
      {\ar@{..}    "4";"5"};
      {\ar@{-}     "5";"6"};
    \end{xy}
    & (n\ge 2)
  \\[2ex]
\Artin{D}_n:\quad
    \begin{xy}
      0;<2em,0em>:<0em,1em>::
      (0,1)*+{0}="0";
      (0,-1)*+{1}="1";
      (1,0)*+{2}="2";
      (2,0)*+{3}="3";
      (3,0)*+{4}="4";
      (3.75,0)="5";
      (4.6,0)="6";
      (6,0)*+{(n-1)}="7";
      {\ar@{-}  "0";"2"};
      {\ar@{-}  "1";"2"};
      {\ar@{-}  "2";"3"};
      {\ar@{-}  "3";"4"};
      {\ar@{-} "4";"5"};
      {\ar@{..}  "5";"6"};
      {\ar@{-}  "6";"7"};
    \end{xy}
    & (n\ge 3;\; \Artin{D}_3 = \Artin{A}_3)
  \\
\Artin{E}_n:\quad
    \begin{xy}
      0;<2em,0em>:<0em,2em>::
      (0,0)*+{1}="1";
      (1,0)*+{2}="2";
      (2,0)*+{3}="3";
      (3,0)*+{4}="4";
      (3.75,0)="5";
      (4.6,0)="6";
      (6,0)*+{(n-1)}="7";
      (2,1)*+{0}="0";
      {\ar@{-}  "1";"2"};
      {\ar@{-}  "2";"3"};
      {\ar@{-}  "3";"4"};
      {\ar@{-}  "4";"5"};
      {\ar@{..} "5";"6"};
      {\ar@{-}  "6";"7"};
      {\ar@{-}  "3";"0"};
    \end{xy}
    & (n=6,7,8)
  \\[2.5ex]
\multicolumn{2}{l}{
\Artin{F}_4:\quad
    \begin{xy}
      0;<2em,0em>:<0em,2em>::
      (0,0)*+{1}="1";
      (1,0)*+{2}="2";
      (2,0)*+{3}="3";
      (3,0)*+{4}="4";
      {\ar@{-}     "1";"2"};
      {\ar@{-}^{4}  "2";"3"};
      {\ar@{-}     "3";"4"};
    \end{xy}
  \qquad
\Artin{H}_3:\quad
    \begin{xy}
      0;<2em,0em>:<0em,2em>::
      (0,0)*+{1}="1";
      (1,0)*+{2}="2";
      (2,0)*+{3}="3";
      {\ar@{-}^{5}  "1";"2"};
      {\ar@{-}     "2";"3"};
    \end{xy}
  \qquad
\Artin{H}_4:\quad
    \begin{xy}
      0;<2em,0em>:<0em,2em>::
      (0,0)*+{1}="1";
      (1,0)*+{2}="2";
      (2,0)*+{3}="3";
      (3,0)*+{4}="4";
      {\ar@{-}^{5}  "1";"2"};
      {\ar@{-}     "2";"3"};
      {\ar@{-}     "3";"4"};
    \end{xy}
}
  \\[2ex]
\multicolumn{2}{l}{
\Artin{I}_2(p):\quad
    \begin{xy}
      0;<2em,0em>:<0em,2em>::
      (0,0)*+{1}="1";
      (1,0)*+{2}="2";
      {\ar@{-}^{p}  "1";"2"};
    \end{xy}
    \qquad (p\ge 3;\; \Artin{I}_2(3) = \Artin{A}_2;\; \Artin{I}_2(4) = \Artin{B}_2)
}
\end{array}\]
\end{theorem}

\subsection{Penetration distance and penetration sequences}\label{S:BackgroundPenetration}

Throughout this section, let $M$ be a Garside monoid with Garside element $\Delta=\Delta_M$, set of atoms $\Atoms=\Atoms_M$, and set of proper simple elements $\pSimples=\pSimples_M$.

In \cite{GT13}, we investigated the \emph{penetration distance}, that is, the number of factors in the normal form of an element $x\in M$ that undergo a non-trivial change when $x$ is multiplied on the right by an element $y\in M$:

\begin{definition}[{\cite[Definition 3.2]{GT13}}]
  For $x,y\in M$, the \emph{penetration distance} for the product
  $xy$ is
  \[
  \pd(x,y) = \cl(x) - \max\big\{ i \in \{0,\ldots,\cl(x)\} :
        x\Delta^{-\inf(x)}  \wedge \Delta^i
        = xy\Delta^{-\inf(xy)} \wedge \Delta^i \big\} \,.
  \]
\end{definition}

For certain probability distributions for $x$ and $y$, we calculated the expected value of $\pd(x,y)$ by analysing the regular languages $\calL$ and $\PSeq = \PSeq_M = \bigcup_{k\in\NN} \PSeq_M^{(k)}$, where
$\PSeq^{(k)}=\PSeq_M^{(k)}$ denotes the set of all \emph{penetration sequences} of length $k$:

\begin{definition}[{\cite[Definition 4.2]{GT13}}]
  A word $(s_k, m_k) \ldot\cdots \ldot(s_2, m_2) \ldot(s_1, m_1)$ in
  $\left(\pSimples \times \pSimples\right)^*$ is a \emph{penetration
    sequence of length $k$} if $m_1 \preccurlyeq \partial s_1$ holds, and if one has
  $s_i m_i \ne \Delta$, $\partial s_{i+1} \wedge s_i = \id$, and
  $m_{i+1} = \partial s_{i+1} \wedge s_{i} m_{i}$ for $i=1,\ldots,k-1$.
\end{definition}

\begin{notation}\label{N:GrowthRates}
 Observe that the regular languages $\calL_M$, $\calLbar_M$ and $\PSeq_M$ are factorial (that is, closed under taking subwords) hence, by \cite[Corollary 4]{Shur08}, there exist constants $p_M,q_M,r_M\in\NN$ and $\alpha_M,\beta_M,\gamma_M\in \{0\}\cup [1,\infty[$ such that one has
 \[
   \big|\PSeq_M^{(k)}\big|\in\Theta(k^{p_M}{\alpha_M}^k) \;,\;\;
   \big|\calL_M^{(k)}\big|\in\Theta(k^{q_M}{\beta_M}^k)\; \text{ and }\;
   \Big|\calLbar_M^{(k)}\Big|\in\Theta(k^{r_M}{\gamma_M}^k) \;.
 \]
\end{notation}

One of the key results of \cite{GT13} was that, if $x\in\calL_M^{(k)}$ and $y\in\Atoms_M$ are chosen with uniform probability, the expected value of the penetration distance is uniformly bounded (that is, there exists a bound that is independent of $k$) if one has $\alpha_M < \beta_M$~\cite[Theorem 4.7]{GT13}.

\subsection{\ZS{} products}\label{S:BackgroundZappaSzep}

\ZS{} products of Garside monoids were considered in \cite{Picantin01} (there called ``crossed products'') and \cite{Zappa-Szep}.  The notion of \ZS{} products generalises those of direct and semidirect products; the fundamental property being the existence of unique decompositions of the elements of a monoid as products of elements of two submonoids.

\begin{definition}[{\cite[Definition~1]{Zappa-Szep}}]\label{D:ZappaSzepProduct}
  Let $M$ be a monoid with two submonoids~$G$ and~$H$.  We say that~$M$ is
  the (internal) \emph{\ZS{} product} of~$G$ and~$H$, written
  $M=G\zs H$, if for every $x\in M$ there exist unique $g_1, g_2\in
  G$ and $h_1, h_2 \in H$ such that $g_1 h_1 = x = h_2 g_2$.
\end{definition}

It was shown in \cite{Zappa-Szep} that the \ZS{} product $M=G\zs H$ admits a Garside structure if and only if both~$G$ and~$H$ do~\cite[Theorem~34, Theorem~37]{Zappa-Szep}.
\bigskip

If Garside structures for $M$, $G$ and $H$ are chosen in a compatible way, then normal forms in~$M$ can be completely described in terms of normal forms in~$G$ and in~$H$:

\begin{definition}[{\cite[Definition~40]{Zappa-Szep}}]\label{D:ZS-GarsideStructure}
  The tuple $(M,G,H,\Delta_M, \Delta_G, \Delta_H)$ is a \emph{\ZS{} Garside structure} if:
  \begin{enumerate}[(a)]
   \item $M=G\zs H$ holds;
   \item $M$, $G$ and $H$ are Garside monoids with Garside elements $\Delta_M$, $\Delta_G$ and $\Delta_H$, respectively; and
   \item $\Delta_M = \Delta_G\Delta_H$ holds.
  \end{enumerate}
\end{definition}

\begin{theorem}[{\cite[Theorem~30, Theorem~41, Corollary~55]{Zappa-Szep}}]\label{T:ZappaSzep}
  Suppose that $(M,G,H,\Delta_M, \Delta_G, \Delta_H)$ is a \ZS{} Garside structure.
 
  Then the following hold.
  \begin{enumerate} \itemsep 0em \vspace{-0.5\topskip}
   \item
     The map $G \times H \to M$ given by $(g,h) \mapsto g \join h$ is a poset isomorphism
     $(G, \prefix_G)\times(H, \prefix_H) \to (M, \prefix_M)$.

   \item
     For all $g \in G$ and $h\in H$, one has
     \begin{align*}
       {\inf}_M(g \join h) &= \min({\inf}_G(g), {\inf}_H(h)) \\
        {\sup}_M(g \join h) &= \max({\sup}_G(g), {\sup}_H(h))
        \;.
     \end{align*}
   \item
     The map $\psi \from \calLbar_G \times \calLbar_H \to \calLbar_M$ given by
     \[
        \psi\big(g_1 | g_2 | \cdots | g_m \,,\, h_1 | h_2 | \cdots | h_n\big)
          = \NF\big( (g_1 g_2 \cdots g_m) \join (h_1 h_2 \cdots h_n) \big)
     \]
     is a bijection.
  \end{enumerate}
\end{theorem}

\section{Essential simple elements}\label{S:Essential}

Throughout this section, let~$M$ be a Garside monoid with Garside element~$\Delta=\Delta_M$,
set of atoms~$\Atoms=\Atoms_M$ and set of proper simple elements~$\pSimples=\pSimples_M$, and let~$\calL=\calL_M$ be the set of words over the alphabet~$\pSimples$ that are in normal form.

By the \emph{acceptor graph}~$\Acceptor=\Acceptor_M$ of~$\calL$ we shall mean the labelled directed graph
with vertex set~$\pSimples$ and, for $x,y\in\pSimples$, an edge labelled~$y$ from~$x$ to~$y$ iff
$\complement x \meet y = \id$.  Paths in this graph are in 1-to-1 correspondence with words in normal form.  This graph can be made into a deterministic finite state automaton (DFA) accepting~$\calL$ by adding an initial vertex~$\id_\Gamma$ with an edge labelled~$y$ from~$\id_\Gamma$ to~$y$ for each $y\in\pSimples$.

\begin{definition}
  Say that a simple element $x \in \pSimples$ is \emph{essential} if
  for all $K \in \NN$ there exists a word $x_1 |x_2 |\cdots |x_n \in
  \calL$ such that $x = x_i$ for some $K < i < n-K$.
\end{definition}
Let $\Ess = \Ess_M$ denote the set of all essential simple elements and let~$\calL_\Ess = \calL_{\Ess,M}$ be the restriction of $\calL_M$ to essential simple elements.
\begin{align*}
 \calL_\Ess := \calL_{\Ess,M} := \calL_M \intersect \Ess^* \\
 \calL_\Ess^{(k)} := \calL_{\Ess,M}^{(k)} := \calL^{(k)}_M \intersect \Ess^*
\end{align*}

\begin{remark}
 A proper simple element is essential if and only if it is contained in a non-singleton strongly connected component of the acceptor~$\Acceptor$ (that is, in a strongly connected component that contains at least one edge).
\end{remark}

\begin{example}
  In the Garside monoid $M = \langle a, b \st aba = b^2\rangle^+$ with Garside element $\Delta = b^3$ one has $\complement b \meet x = b^2 \meet x \ne
  \id$ for all proper simple elements $x\in\pSimples$, whence $b$ can
  only occur as the final simple factor and $b^2$ can only occur as
  the first simple factor in a word in normal form.  That is, the
  simple elements $b$ and $b^2$ are not essential.
\end{example}


\begin{proposition} \label{prop:complement-essential}
  A proper simple element $x\in\pSimples$ is essential if and only if
  its complement $\complement x$ is essential.
\end{proposition}
\begin{proof}
  Suppose that $x\in\pSimples$ is essential and that we are given
  $K\in\NN$.  As $x$ is essential there exists a word 
  \[ w_n |\cdots |w_2| w_1| x| y_1| y_2| \cdots |y_m \in \calL \]
  with $n, m > K$.  As for any $s,t\in\Simples$ and any $k\in\ZZ$ one has $\complement s\meet t=\id$ iff $\complement(\complement^{2k+1} t)\meet \complement^{2k+3} s=\id$, the word
  \[ \complement^{-2m+1} y_m \ldot \cdots \ldot \complement^{-3} y_2 \ldot
     \complement^{-1} y_1 \ldot \complement x \ldot \complement^3 w_1 \ldot
     \complement^5 w_2 \ldot \cdots \ldot \complement^{2n+1} w_n \]
  is also in normal form and hence lies in $\calL$.  Therefore
  $\complement x$ is essential.

  Applying the same argument with $\rightcomplement = \complement^{-1}$, we deduce
  that if $\complement x$ is essential then $x$ is essential.
\end{proof}

\begin{lemma} \label{lemma:non-minimal-essential}
  A simple element in $M(k)$ is essential if and only if its normal
  form in $M$ is a length $k$ word of essential simple elements.
  \[ \Ess_{M(k)} = \calL_{\Ess,M}^{(k)} \]
\end{lemma}
\begin{proof}
  First we will show that the $M$-normal form of every
  $M(k)$-essential element is a length $k$ word of $M$-essential
  elements.  So suppose that $x\in\Ess_{M(k)}$ is $M(k)$-essential
  and that $\NF_M(x) = x_1 |x_2 |\cdots |x_\ell$ is the $M$-normal form of
  $x$.
  If $\ell<k$, then $\Delta\prefix \complement_{M(k)}x$ and thus $x\ldot y$ cannot be in normal form for any non-trivial $y\in \Simples_{M(k)}$, contradicting the choice of~$x$.
  Similarly, if $x_1 = \Delta$, then $y\ldot x$ cannot be in normal form for any proper $y\in \pSimples_{M(k)}$, contradicting the choice of~$x$.
  Hence we have $\ell=k$ and $x_i\in\pSimples_M$ for $i=1,\ldots,k$.
  
  As~$x$ is essential, for every $K \in \NN$ there exists a word
  \[ w_{K+1} |\cdots| w_2| w_1|x|y_1| y_2 |\cdots| y_{K+1} \in \calL_{M(k)}. \] 
  By \autoref{lemma:non-minimal-normal-form} we can replace each $w_i$
  and each $y_i$ by their $M$-normal forms to produce a word in
  $M$-normal form.
  \begin{align*}
    \NF_M(w_i) &= w_{i,1} |w_{i,2} |\cdots| w_{i,k} \\
    \NF_M(y_i) &= y_{i,1} |y_{i,2} |\cdots| y_{i,k} 
  \end{align*}
  As all $w_i$ and $y_i$ are proper $M(k)$-simple elements we have
  that, for $i \le K$, the elements $w_{i,j}$ and $y_{i,j}$ are proper $M$-simple
  elements.  Hence
  \begin{align*}
  w_{K,1} |\cdots |w_{K,k}|\cdots|
  w_{1,1} |\cdots |w_{1,k}| 
  x_1 |\cdots |x_k| 
   & y_{1,1} |\cdots |y_{1,k}|\cdots  |y_{K,1} |\cdots| y_{K,k}
  \end{align*}
  is a word in $\calL_M$ and so each $x_i$ is essential.

  It remains to show that every $x_1 |x_2 |\cdots| x_k \in
  \calL_{\Ess,M}^{(k)}$ defines an $M(k)$-essential simple element.
  Using \autoref{lemma:normal-form-non-minimal} it is clear that $x_1
  x_2 \cdots x_k$ is a proper $M(k)$-simple element.  As $x_1$ and
  $x_k$ are essential, for any $K\in\NN$ there exist words
  $w_{Kk} |\cdots| w_2 |w_1 \in \calL_M$ and $y_1| y_2| \cdots| y_{Kk} \in \calL_M$
  such that $w_1 | x_1$ and $x_k | y_1$.  By
  \autoref{lemma:normal-form-non-minimal}, if we let for $i=1,\ldots,K$
  \begin{align*}
    \bar w_i &= w_{ik} w_{ik-1} \cdots w_{(i-1)k + 1} \\
    \bar x   &= x_1 x_2 \cdots x_k \\
    \bar y_i &= y_{(i-1)k+1} y_{(i-1)k+2} \cdots y_{ik}
  \end{align*}
  then 
  \[
  \bar w_K \ldot\cdots\ldot \bar w_2\ldot \bar w_1\ldot \bar x\ldot \bar y_1\ldot \bar y_2\ldot \cdots \ldot\bar y_K
  \]
  is in $M(k)$-normal form.  Moreover, because all $w_i$ and all
  $y_i$ are proper $M$-simple elements, we know that each letter of
  this word is a proper $M(k)$-simple element.  Therefore this word
  is in $\calL_{M(k)}$ and hence $\bar x$ is essential.
\end{proof}

\begin{proposition}
  If $\Ess = \emptyset$ then $M = \NN$.
\end{proposition}
\begin{proof}
  First note that $\calL$ must be finite:  Otherwise, there would be a word $w \in \calL$ whose length is longer than the number of states in the automaton accepting~$\calL$ whence, by the pumping lemma, there would exist words $x$, $y$, $z$, where $y \ne \epsilon$,
such that $w = x\ldot y\ldot z$ and $x\ldot y \ldot \cdots \ldot y \ldot z \in \calL$ for any number of copies of~$y$; in particular, every letter of~$y$ would be essential in contradiction to the hypothesis.
  
  By \autoref{lemma:non-minimal-essential}, if $\Ess_M$ is empty then $\Ess_{M(k)}$ is empty for all $k$.  Hence, passing to $M(k)$, where $k$ is the length of the longest word contained in $\calL$, we may assume that all the words in $\calL$ have length 1.
  
  For every $s \in \pSimples$ and every $a \in\Atoms$ we have $sa \in \Simples$, as otherwise, $s|a$ would be a word of length $2$ in $\calL$.
  Choosing any $\prefix$--maximal element $s \in \pSimples$, we have $s a = \Delta$ for all atoms $a \in \Atoms$, hence, by cancellativity, all the atoms must be equal.
  In other words there is a single atom and so $M = \NN$.
\end{proof}

\begin{proposition}
  If $\Ess \ne \emptyset$ then $|\Ess| > 1$.
\end{proposition}
\begin{proof}
  Suppose $\Ess = \{ s \}$.

  As $s$ is essential there exist arbitrarily long words containing
  $s$.  So there is a word $w \in \calL$ whose length is longer than
  the number of states in the automaton accepting~$\calL$ hence, by
  the pumping lemma, there exist words $x$, $y$, $z$, where $y \ne \epsilon$,
  such that $w = x\ldot y\ldot z$ and $x\ldot y \ldot \cdots \ldot y \ldot z \in \calL$ for any number of copies of~$y$.

  This means that every letter of $y$ is essential.  Therefore $y$ is
  a positive power of $s$ and so $s|s$, that is, $\complement s \meet s
  = \id$.

  By \autoref{prop:complement-essential}, $\complement s$ is
  essential, hence $\complement s = s$.  Therefore, $\complement s
  \meet s = s$, which is a contradiction.
\end{proof}

\begin{definition}
  Say that the language $\calL$, or simply~$M$, is \emph{essentially $k$-transitive}
  if $\Ess \ne \emptyset$ and for all $x, y \in \Ess$ there exists a word of
  length less than or equal to $k+1$ in~$\calL_\Ess$ which starts with~$x$ and ends with~$y$.

  Say that $\calL$ is essentially transitive if it is essentially
  $k$-transitive for some~$k$.
\end{definition}


\begin{remark} 
The language $\calL$ is essentially transitive if and only if the acceptor graph~$\Acceptor$ has exactly one non-singleton strongly connected component (that is, exactly one strongly connected component containing at least one edge).

In this case, the language $\calL$ is essentially $k$-transitive iff the diameter of the non-singleton strongly connected component of~$\Acceptor$ is at most~$k$.

\autoref{prop:complement-essential} implies that if the language~$\calL$ is essentially $k$-transitive, then one has $k\ge 2$.
\end{remark}

\begin{example}\label{example:NotEssentiallyTransitive}
 The monoid $M = \langle a,b,c \mid  ab=c^2, bc=a^2, ca=b^2 \rangle$ is a Garside monoid with Garside element $\Delta = a^3 = b^3 = c^3 = abc = bca = cab$.  The lattice of simple elements and the acceptor for the language~$\calL_M$ are shown in \autoref{F:NotEssentiallyTransitive}.
 
 Every proper simple element of~$M$ is essential, but~$M$ is not essentially transitive.
  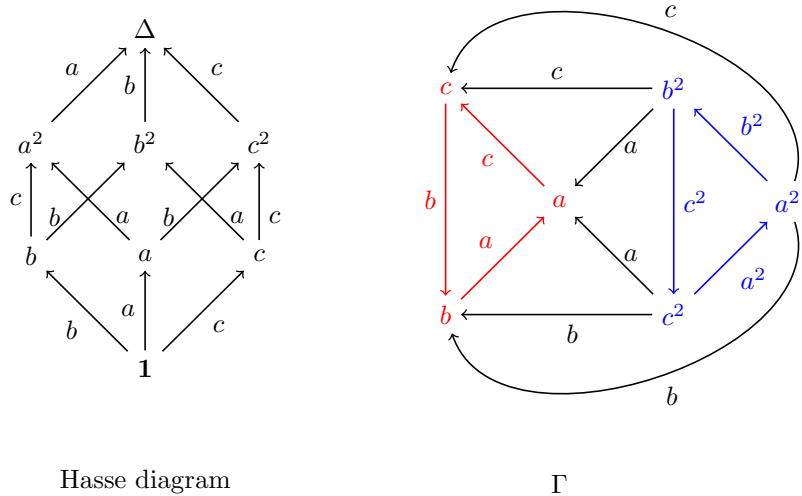
\begin{figure}
  \hfill
  \parbox[b]{0.4\textwidth}{
    \begin{tikzpicture}[->,auto,scale=1.5, semithick]
      \node (1)  at (1,0) {$\id$};
      \node (b)  at (0,1) {$b$};
      \node (a)  at (1,1) {$a$};
      \node (c)  at (2,1) {$c$};
      \node (a2) at (0,2) {$a^2$};
      \node (b2) at (1,2) {$b^2$};
      \node (c2) at (2,2) {$c^2$};
      \node (D)  at (1,3) {$\Delta$};

      \path (1)  edge node {$a$} (a);
      \path (1)  edge node {$b$} (b);
      \path (1)  edge node [below right] {$c$} (c);
      \path (a)  edge node [pos=0.3, right] {$a$} (a2);
      \path (a)  edge node [pos=0.3, left]{$b$} (c2);
      \path (b)  edge node [pos=0.3, left]{$b$} (b2);
      \path (b)  edge node [left]{$c$} (a2);
      \path (c)  edge node [pos=0.3, right] {$c$} (c2);
      \path (c)  edge node [pos=0.3, right] {$a$} (b2);
      \path (a2) edge node {$a$} (D);
      \path (b2) edge node {$b$} (D);
      \path (c2) edge node [above right] {$c$} (D);
     
      \node at (1,-1) {Hasse diagram};
    \end{tikzpicture}
  } \hfill
  \parbox[b]{0.46\textwidth}{
    \begin{tikzpicture}[->,auto,scale=1.5, semithick]
      \node [red] (b)  at (0,0) {$b$};
      \node [red] (c)  at (0,2) {$c$};
      \node [red] (a)  at (1,1) {$a$};
      \node [blue] (c2) at (2,0) {$c^2$};
      \node [blue] (b2) at (2,2) {$b^2$};
      \node [blue] (a2) at (3,1) {$a^2$};

      \path (a) edge [red] node [below left] {$c$} (c);
      \path (c) edge [red] node [left]{$b$} (b);
      \path (b) edge [red] node [above left] {$a$} (a);
      \path (a2) edge [blue] node [above right] {$b^2$} (b2);
      \path (b2) edge [blue] node [right]{$c^2$} (c2);
      \path (c2) edge [blue] node [below right] {$a^2$} (a2);
      \path (a2) edge [bend left=90] node [below right] {$b$} (b);
      \path (a2) edge [bend right=90] node [above right] {$c$} (c);
      \path (b2) edge node [right]{$a$} (a);
      \path (b2) edge node [above]{$c$} (c);
      \path (c2) edge node [right] {$a$} (a);
      \path (c2) edge node [below right] {$b$} (b);

      
      \node at (1,-1.5) {$\Acceptor$};
    \end{tikzpicture}
  } \hfill{}
  \caption{The lattice of simple elements and the acceptor for the regular language of words in normal form for the monoid from \autoref{example:NotEssentiallyTransitive}; the strongly connected components of the acceptor are coloured.}
  \label{F:NotEssentiallyTransitive}
  \end{figure}
\end{example}

\begin{proposition}
  For any pair of Garside monoids $G$ and $H$, each containing at
  least one proper simple element, their free product amalgamated over their Garside elements,
  $G*_{\Delta_G = \Delta_H} H$, is a Garside monoid with Garside element $\Delta_G = \Delta_H$ for which
  every proper simple element is essential and whose language of
  normal forms is essentially $2$-transitive.
\end{proposition}
\begin{proof}
  The amalgamated product $G*_{\Delta_G = \Delta_H} H$ is a Garside
  monoid \cite[Prop.~5.3]{DehornoyParis99}\cite{Froehle07} whose set of
  proper simple elements is the disjoint union of the proper simple
  elements of $G$ and $H$.%

  Now consider $u,v \in \pSimples$.
  If $u$ and $v$ lie in different factors then $u|v$ is in normal form.
  If $u$ and $v$ lie in the same factor, then choosing any proper simple
  element $w$ from the other factor yields a word $u| w| v$ that is in
  normal form.
\end{proof}

\begin{lemma}\label{lemma_essential_transitivity_framed}
  If the language of normal forms of $M(k)$ is essentially transitive
  then the language of normal forms of $M$ is essentially transitive
\end{lemma}
\begin{proof}
  As $\Ess_{M(k)}\ne\emptyset$, we have $\Ess_M\ne\emptyset$ by \autoref{lemma:non-minimal-essential}.
  Let $x,y \in \Ess_M$.  Since $x$ and~$y$ are essential, there exist elements $x_1,\ldots,x_{k-1}$ and 
  $y_2,\ldots,y_k$ in~$\Ess_M$ such that $ x_1 \ldot \cdots\ldot x_{k-1} \ldot x$ and
  $y \ldot y_2 \ldot \cdots \ldot y_k$ are in normal form.
  By \autoref{lemma:non-minimal-essential}, $\bar x = x_1 \cdots x_{k-1} x$ and $\bar y = y y_2 \cdots y_k$ are elements of~$\Ess_{M(k)}$.
  Now as the language of normal forms in $M(k)$ is
  essentially transitive, we can find a path connecting~$\bar x$ to~$\bar y$ and then, by \autoref{lemma:non-minimal-normal-form},
  taking the $M$-normal form of each letter gives a path in~$\Acceptor_M$ from~$x$ to~$y$.
\end{proof}

\begin{remark}
The converse of \autoref{lemma_essential_transitivity_framed} does not hold.  For example, consider the monoid $M=\langle a, b \st a^2=b^2\rangle^+$ (see \autoref{example:aa=bb}).
$M$ is essentially transitive, but $M(2)$ is not:  We have $\Ess_M = \{a,b\}$ and $a|b$ as well as $b|a$.  However, $\Ess_{M(2)} = \{ab,ba\}$ and $\calL_{M(2)} = (ab)^* \cup (ba)^*$.
\end{remark}

\begin{lemma}
  If the language of normal forms of $M$ is essentially transitive then
  for all~$k_0$ there exists $k \ge k_0$ such that the language of
  normal forms of $M(k)$ is essentially transitive.
\end{lemma}
\begin{proof}
  First note that $M(k)$ is essentially transitive if and only if for
  each pair $x$, $y$ of $M$-essential elements there exists a
  path $x|x_1|x_2|\cdots|x_\ell|y$ such that $\ell$ is a multiple of $k$.

  For each pair $x$, $y$ of $M$-essential elements choose
  $z\in\Ess_M$ together with paths $x|x_1|x_2|\cdots|x_\ell|z$,
  $z|z_1|z_2|\cdots|z_m|z$ and $z|y_1|y_2|\cdots|y_n|y$.

  As the set of essential elements is finite, we can choose an integer $k\ge k_0$ such that, for each pair $x$, $y$ the integer $m$ is coprime to $k$.
  Then for each pair $x$,~$y$ there exists $p$ such that $\ell + pm + n = 0 \mod{k}$.
\end{proof}

\begin{definition}
For an integer $k>0$, a word $x_1|\cdots|x_k\in\calL^{(k)}$ is called \emph{rigid}, if the word $x_k|x_1$ is in normal form.
Let $\calL_\Rig^{(k)}$ denote the set of rigid words in~$\calL^{(k)}$.
\end{definition}

\begin{theorem}[{\cite[Proposition~4.1]{Caruso13}}] \label{PercentageOfRigid}
  If the language $\calL$ is essentially transitive then, for sufficiently large~$\ell$, the proportion of rigid elements in the ball of radius $\ell$, that is in $\Union_{k\le \ell}\calL^{(k)}$, is bounded below by a positive constant.
\end{theorem}
\begin{proof}
A word $x_1|\cdots|x_k \in \calL$ is rigid if and only if the word $x_1|\cdots|x_k\ldot x_1$ is in normal form, that is, traces out a cycle in $\Acceptor$.

Assume that the language~$\calL$ is essentially $D$-transitive.  Then, for any word
$w = x_1|\cdots|x_\ell \in \calL_{\Ess}^{(\ell)}$, there exist an integer $c<D$ and $x_{\ell+1},\ldots,x_{\ell+c} \in \pSimples$ such that
$x_1|\cdots|x_\ell|x_{\ell+1}|\cdots|x_{\ell+c}|x_1 \in \calL^{(\ell)}$.
Hence, any $w \in \calL_{\Ess}^{(\ell)}$ can be extended to a rigid word $w' \in \calL_{\Rig}^{(\ell')}$ for some $\ell'\in\{\ell,\ldots,\ell+D-1\}$.
For fixed $\ell$, the map $w\mapsto w'$ is clearly injective, so we have
$\Big|\Union_{k=\ell}^{\ell+D-1} \calL_{\Rig}^{(k)}\Big| \ge \big|\calL_{\Ess}^{(\ell)}\big|$.

Now let
$
   a_d := \Big| \Union_{k=(d-1)D+1}^{dD} \calL_{\Rig}^{(k)} \Big|
   \quad\text{ and }\quad
   b_d := \Big| \Union_{k=(d-1)D+1}^{dD} \calL_{\Ess}^{(k)} \Big|
   \;.
$
The language~$\calL_{\Ess}$ is factorial and, as there are essential elements, infinite.  Thus, by~\cite[Corollary~4]{Shur08}, we have $\big|\calL_{\Ess}^{(k)}\big|\in\Theta(k^{q'}{\beta'}^k)$ with $\beta'\ge 1$.
In particular, $a_d, b_d\ge 1$ holds for all values of~$d$.
Moreover, as $\sum_{k=\ell}^{\ell+D-1} k^{q'} {\beta'}^k \in \Theta\big(\ell^{q'} {\beta'}^\ell\big)$, we have $\Theta(a_d) = \Theta(b_d) = \Theta(b_{d+1}) =  \Theta\big(d^{q'} {\beta'}^{dD}\big)$.  Hence, there exists a constant $C>0$ such that we have $a_d > \frac1C b_d$ and $b_{d+1} < C b_d$ for sufficiently large~$d$.
Hence,
\[
  \frac{\,\Big|\Union_{k=0}^\ell \calL_{\Rig}^{(k)}\Big|\,}{\,\Big|\Union_{k=0}^\ell \calL_{\Ess}^{(k)}\Big|\,}
  \ge \frac{a_1+\cdots+a_{\lfloor\frac{\ell}{D}\rfloor}}
           {1+b_1+\cdots+b_{\lfloor\frac{\ell}{D}\rfloor}+b_{\lfloor\frac{\ell}{D}\rfloor+1}}
  > \frac1{2C^2}
\]
holds for sufficiently large $\ell$.

So it remains to show that there exists a constant~$K'$ such that, for sufficiently large~$\ell$, one has $\Big|\Union_{k=0}^\ell \calL_{\Ess}^{(k)}\Big| \ge \frac{1}{K'} \Big|\Union_{k=0}^\ell \calL^{(k)}\Big|$.

For $w=x_1|\cdots|x_\ell$, let~$\alpha(w)$ be the minimal integer satisfying $x_{\alpha(w)}\in\Ess$, or~$0$ if no such integer exists.  Similarly, let~$\omega(w)$ be the maximal integer satisfying $x_{\omega(w)}\in\Ess$, or~$0$ if no such integer exists.
Clearly, if~$\alpha(w)>0$ and $\omega(w)>0$ hold, then one has $x_{\alpha(w)}|\cdots|x_{\omega(w)}\in \calL_{\Ess}$.
By the pumping lemma, $\alpha(w)>0$ and $\omega(w)>0$ hold for all but a finite number~$\widetilde{K}$ of words~$w$.
Moreover, there exists a constant~$K$ such that for any $x\in\Ess$ both, the number of words $y_1|\cdots|y_r$ such that $y_1,\ldots,y_r\in\pSimples\setminus\Ess$ and $y_r|x$ hold, and the number of words $z_1|\cdots|z_s$ such that $z_1,\ldots,z_s\in\pSimples\setminus\Ess$ and $x|z_1$ hold, are bounded above by~$K$.
Hence, for words~$w$ satisfying $\alpha(w)>0$ and $\omega(w)>0$, the map $x_1|\cdots|x_\ell \mapsto x_{\alpha(w)}|\cdots|x_{\omega(w)}$ is at most $K^2$-to-$1$.  Thus, one has
$\Big|\Union_{k=0}^\ell \calL^{(k)}\Big| \le \widetilde{K} + K^2 \Big|\Union_{k=0}^\ell \calL_{\Ess}^{(k)}\Big|$, proving the claim.
\end{proof}

\comment{\begin{itemize}
\itemS{Can we extend this to Zappa-Szép products of essentially transitive monoids?  I guess all this would need is a description of the rigid elements of a $\zs$-product in terms of those of the factors.  I don't think we considered that in the previous paper.}
\end{itemize}}

\begin{example} \label{example:aa=bb}
  Under the hypotheses of \autoref{PercentageOfRigid}, it is not necessarily true that the percentage of rigid elements in $\calL^{(k)}$ is bounded below by a positive constant for sufficiently large values of~$k$.  In other words, it is possible for the sphere of radius $k$ to contain an arbitrarily small proportion of rigid elements.

As an example, consider the monoid $M=\langle a, b \st a^2 = b^2 \rangle^+$ with Garside element $\Delta = a^2 = b^2$; see \autoref{F:aa=bb}.
  \begin{figure}
  \hfill
  \parbox[c]{0.45\textwidth}{
    \begin{xy}
      0;<3em,0em>:<0em,3em>::
      (1,-0.6)*+{\text{Hasse diagram}};
      (1,0)*+{\id}="e";
      (0,1)*+{a}="a";
      (2,1)*+{b}="b";
      (1,2)*+{\Delta}="aa";
      {\ar@{->}^{a} "e";"a"};
      {\ar@{->}_{b} "e";"b"};
      {\ar@{->}^{a} "a";"aa"};
      {\ar@{->}_{b} "b";"aa"};
    \end{xy}
  } \hfill
  \parbox[c]{0.45\textwidth}{
    \begin{xy}
      0;<3em,0em>:<0em,3em>::
      (0.5,-0.85)*+{\Acceptor};
      (0,0)*+{a}="a";
      (1,0)*+{b}="b";
      {\ar@/^/^{b} "a";"b"};
      {\ar@/^/^{a} "b";"a"};
    \end{xy}
  } \hfill {}
  \caption{The lattice of simple elements and the acceptor for the regular language of words in normal form for the monoid from \autoref{example:aa=bb}.}
  \label{F:aa=bb}
  \end{figure}
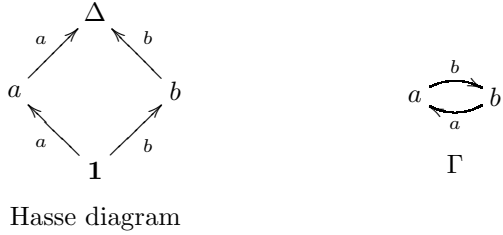
  Clearly, the language $\calL_M$ is essentially transitive.
  (And, moreover, every proper simple element is essential.)
  However, for odd $k=2m+1$, one has
  $\calL_M^{(k)} = \{ a\ldot(b\ldot a)^m , b\ldot(a\ldot b)^m \}$, and so there are no rigid elements of length~$k$.
  \commV{The only problem is that the language of rigid words is not factorial, and thus its growth is not necessarily of the form $\Theta(k^r \rho^k)$; we only have $O(...)$, not $\Theta(...)$.  The example shows why: there can be periodicity effects.  It should be possible to exclude these if the acceptor is primitive, or something like that; in that case, we could say something about $\calL^{(k)}$.}
\end{example}

%
%

\medskip

Being essentially transitive is a rather strong property.  In
particular, it implies that the monoid in question cannot be
decomposed as a \ZS{} product.  As we shall see in
\autoref{example:aa=bb:2-framed}, the converse of this does not hold,
i.e.\ not all monoids that are $\zs$-indecomposable are essentially
transitive.

\begin{proposition}\label{product-not-transitive}
  If $M = G \zs H$ for two submonoids~$G$ and~$H$, then~$M$ is not essentially transitive.
\end{proposition}
\begin{proof}
  By \cite[Theorem~34]{Zappa-Szep},~$G$ and~$H$ are parabolic submonoids of~$K$.
  The corresponding Garside elements~$\Delta_G$ of~$G$ and~$\Delta_H$ of~$H$ are proper simple elements of~$K$.
  Moreover, these elements are essential, as the words $\Delta_G \ldot \cdots \ldot \Delta_G$ and $\Delta_H \ldot \cdots \ldot \Delta_H$ are in normal form for any number of copies of~$\Delta_G$ respectively~$\Delta_H$.
  Yet, by \cite[Corollary~48]{Zappa-Szep}, there cannot be a normal form word connecting~$\Delta_G$ to~$\Delta_H$.
\end{proof}

\begin{definition}\label{D:DeltaPure}
  For $y\in M$, let $\Delta_y := \Join\{ x\under y : x \in M \}$.
  The monoid $M$ is called \emph{$\Delta$-pure} if $\Delta_a=\Delta_b$ holds for any $a,b\in\Atoms$.
\end{definition}

\begin{theorem}[{\cite[Proposition~4.7]{Picantin01}\cite[Theorem~39]{Zappa-Szep}}]\label{T:DeltaPure}
  A Garside monoid~$M$ is $\Delta$-pure if and only if it is $\zs$-indecom\-posable.
\end{theorem}

\begin{corollary}\label{TransitiveImpliesDeltaPure}
  If $M$ is essentially transitive then it is $\Delta$-pure.
\end{corollary}
\begin{proof}
  The claim follows with \autoref{product-not-transitive} and \autoref{T:DeltaPure}.
\end{proof}

The following example shows that there exists a Garside monoid that is $\Delta$-pure, hence $\zs$-indecomposable, but not essentially transitive, that is, that the converses to \autoref{product-not-transitive} and \autoref{TransitiveImpliesDeltaPure} do not hold.

\begin{example}\label{example:aa=bb:2-framed}
  Consider the monoid $M = \langle a, b \st a^2 = b^2 \rangle^+$ with Garside element $\Delta = a^4$.
  The lattice of simple elements and the acceptor for the regular language of words in normal form are shown in \autoref{F:aa=bb:2-framed}.

  \begin{figure}
  \hfill
  \parbox[c]{0.45\textwidth}{
    \begin{xy}
      0;<3em,0em>:<0em,3em>::
      (1,-0.6)*+{\text{Hasse diagram}};
      (1,0)*+{\id}="e";
      (0,1)*+{a}="a";
      (2,1)*+{b}="b";
      (-1,2)*+{ab}="ab";
      (1,2)*+{a^2}="aa";
      (3,2)*+{ba}="ba";
      (0,3)*+{a^3}="aaa";
      (2,3)*+{b^3}="bbb";
      (1,4)*+{\Delta}="Delta";
      {\ar@{->}^{a} "e";"a"};
      {\ar@{->}_{b} "e";"b"};
      {\ar@{->}_{a} "a";"aa"};
      {\ar@{->}^{b} "a";"ab"};
      {\ar@{->}^{b} "b";"aa"};
      {\ar@{->}_{a} "b";"ba"};
      {\ar@{->}^{b} "ab";"aaa"};
      {\ar@{->}_{a} "aa";"aaa"};
      {\ar@{->}^{b} "aa";"bbb"};
      {\ar@{->}_{a} "ba";"bbb"};
      {\ar@{->}^{a} "aaa";"Delta"};
      {\ar@{->}_{b} "bbb";"Delta"};
    \end{xy}
  } \hfill
  \parbox[c]{0.45\textwidth}{
    \begin{xy}
      0;<3em,0em>:<0em,3em>::
      (1,-2)*+{\Acceptor};
      (0,0)*+{ab}="ab";
      (2,0)*+{ba}="ba";
      (1,1)*+{a^2}="aa";
      (0,1)*+{a}="a";
      (2,1)*+{b}="b";
      (0,2)*+{b^3}="bbb";
      (2,2)*+{a^3}="aaa";
      {\ar@{->}_{a} "ab";"a"};
      {\ar@{->}^{b} "ba";"b"};
      {\ar@{->}^{a} "bbb";"a"};
      {\ar@{->}_{b} "aaa";"b"};
      {\ar@(d,r)_{ab} "ab";"ab"};
      {\ar@(l,d)_{ab} "bbb";"ab"};
      {\ar@(d,l)^{ba} "ba";"ba"};
      {\ar@(r,d)^{ba} "aaa";"ba"};
    \end{xy}
  } \hfill {}
  \caption{The lattice of simple elements and the acceptor for the regular language of words in normal form for the monoid from \autoref{example:aa=bb:2-framed}.}
  \label{F:aa=bb:2-framed}
  \end{figure}
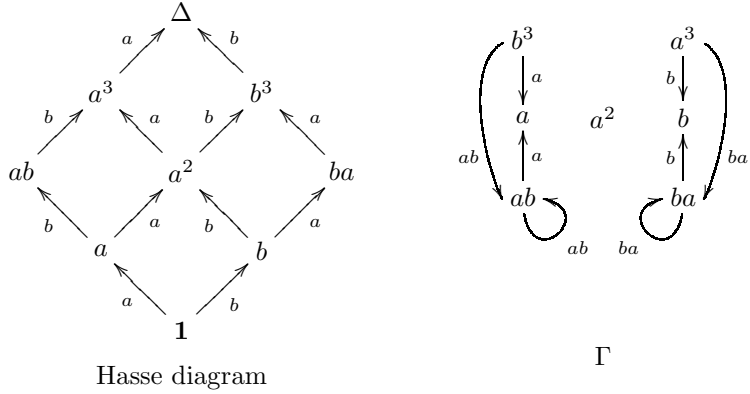

  We see that $ab$ and $ba$ are essential elements that are in different strongly connected components of~$\Acceptor$, and so~$\calL$ is not essentially transitive.

  Now consider $\Delta_a = \Join \{ x \under a : x \in M \}$.
  If $x$ has $a$ as a prefix then $x \under a = \id$, so we can
  restrict our attention to elements which do not have $a$ as a
  prefix.  This means that $x = (ba)^k$ or $x = (ba)^k b$ for some
  $k$.  In the first case $(ba)^k \join a = (ba)^k a$, so $x \under a
  = a$.  In the second case $(ba)^k b \join a = (ba)^k b^2$, so $x
  \under a = b$.  Therefore $\Delta_a = a \join b = a^2$.  Similarly,
  $\Delta_b = a^2$.  Hence $M$ is $\Delta$-pure, and thus $\zs$-indecomposable.
\end{example}

\section{Growth rates}\label{S:GrowthRates}

Throughout this section, let~$M$ be a Garside monoid with Garside element~$\Delta$, set of atoms~$\Atoms$ and set of proper simple elements~$\pSimples=\Div(\Delta)\setminus\{\id,\Delta\}$.
Recall that~$\alpha_M$ is the exponential growth rate of the regular language~$\PSeq_M$ and that~$\beta_M$ is the exponential growth rate of the regular language~$\calL_M$.

The main aim of this section is to show that $\alpha_M < \beta_M$ holds if the language~$\calL_M$ is essentially transitive and every element of~$\pSimples$ is essential.
In particular, by \cite[Theorem~4.7]{GT13}, the expectation $\Expect_{\nu_k\times\mu_\Atoms}[\pd]$ of the penetration distance $\pd(x,a)$ is in this case bounded independently of~$k$, where~$\nu_k$ is the uniform distribution on~$\calL_M^{(k)}$ and~$\mu_\Atoms$ is the uniform distribution on the set~$\Atoms$.

Moreover, we will show that the expectation of the penetration distance $\Expect_{\nu_k\times\mu_\Atoms}[\pd]$, with~$\nu_k$ and~$\mu_\Atoms$ as above, diverges if~$(M,G,H,\Delta_M, \Delta_G, \Delta_H)$ is a \ZS{} Garside structure such that $\beta_G,\beta_H>1$ holds and at least one of the factors has the property that its language of normal forms is essentially transitive and all proper simple elements are essential.

\begin{theorem}\label{T:BoundedExpectedPD}
If every proper simple element of~$M$ is essential and the language~$\calL_M$ of normal forms is essentially transitive, then one has $\alpha_M<\beta_M$.
\end{theorem}
\begin{proof}
As there is only one monoid, we drop the subscript $M$.

Consider the acceptor $\Acceptor$ of $\calL \subseteq (\pSimples)^*$, whose adjacency matrix is given~by
\[
 \Acceptor_{s_1,s_2} = \begin{cases}
                         1 & \partial s_2\wedge s_1 = \id \\
                         0 & \text{otherwise}
                       \end{cases}
\]
for $s_1,s_2\in\pSimples$.  The growth rate $\beta$ of $\calL$ is the Perron--Frobenius eigenvalue of the non-negative matrix $(\Acceptor_{s_1,s_2})_{s_1,s_2\in\pSimples}$.
Let $x = (x_s)_{s\in\pSimples}$ be an eigenvector for the eigenvalue $\beta$ of $\Acceptor$.

The acceptor $\PSeqAcceptor$ of $\PSeq \subseteq \mathcal{P}^*$, where $\mathcal{P}=\{(s,m)\in\pSimples\times\pSimples : sm\preccurlyeq\Delta\}$, has the adjacency matrix given by
\[
 \PSeqAcceptor_{(s_1,m_1),(s_2,m_2)} = \begin{cases}
                         1 & \partial s_2\wedge s_1 = \id
                             \,\text{ and }\, s_1m_1\ne \Delta
                             \,\text{ and }\, m_2 = \partial s_2\wedge s_1m_1 \\
                         0 & \text{otherwise}
                       \end{cases}
\]
for $(s_1,m_1),(s_2,m_2)\in\mathcal{P}$.  The growth rate $\alpha$ of $\PSeq$ is the Perron--Frobenius eigenvalue of the non-negative matrix $\big(\PSeqAcceptor_{(s_1,m_1),(s_2,m_2)}\big)_{(s_1,m_1),(s_2,m_2)\in\mathcal{P}}$, whence one has
\[
   \alpha = \inf_{z\in(\RR^+)^{\mathcal{P}}} \;
               \max_{t\in\mathcal{P}} \;
                  \frac{(\PSeqAcceptor z)_t}{z_t}
\]
by~\cite[Theorem~3.1]{TamWu89}.
In order to prove the theorem, it is thus sufficient to construct a vector
$y = (y_t)_{t\in\mathcal{P}}$ such that, for any 
$t\in\mathcal{P}$, one has $y_t>0$ and $(\Pi y)_t < \beta y_t$.

To do this, consider
$\widetilde{\mathcal{P}} = \{(s,m)\in\pSimples\times(\pSimples\cup\{\id\}) : sm\preccurlyeq\Delta\}$,
and define a directed graph~$\widetilde{\PSeqAcceptor}$ with vertex set~$\widetilde{\mathcal{P}}$
via its adjacency matrix given by
\[
 \widetilde{\PSeqAcceptor}_{(s_1,m_1),(s_2,m_2)} = \begin{cases}
                         1 & \partial s_2\wedge s_1 = \id \,\text{ and }\, m_2 = \partial s_2\wedge s_1m_1 \\
                         0 & \text{otherwise}
                       \end{cases}
\]
for $(s_1,m_1),(s_2,m_2)\in\widetilde{\mathcal{P}}$.
Observe that~$\widetilde{\PSeqAcceptor}$ has~$\PSeqAcceptor$ as a subgraph and that, locally,~$\widetilde{\PSeqAcceptor}$ resembles the graph~$\Acceptor$:
The edges ending in the vertex $(s_1,m_1)\in\widetilde{\mathcal{P}}$ of $\widetilde{\PSeqAcceptor}$ are in bijection to the edges ending in the vertex~$s_1$ of~$\Acceptor$.  More precisely, for given $s_1,s_2\in\pSimples$ and $m_1\in\pSimples\cup\{\id\}$, there exists an~$m_2\in\pSimples\cup\{\id\}$ such that there is an edge $(s_2,m_2) \to (s_1,m_1)$ in~$\widetilde{\PSeqAcceptor}$, if and only if there is an edge $s_2 \to s_1$ in~$\Acceptor$, and if so,~$m_2 = \partial s_2\wedge s_1 m_1$ is uniquely determined, that is, there exists exactly one such edge.
(See \autoref{F:Perron-Frobenius-aa=bb} and \autoref{F:Perron-Frobenius-aba=bab}.)

\begin{figure}
  \hfill
  \parbox[b]{0.3\textwidth}{
    \begin{tikzpicture}[->,auto,node distance=2cm, semithick]
      \node (1) at (1,0) {$\id$};
      \node (a) at (0,1) {$a$};
      \node (b) at (2,1) {$b$};
      \node (D) at (1,2) {$\Delta$};

      \path (1) edge node               {$a$} (a);
      \path (1) edge node [below right] {$b$} (b);
      \path (a) edge node               {$a$} (D);
      \path (b) edge node [above right] {$b$} (D);

      \node at (1,-1) {Hasse diagram};
    \end{tikzpicture}
  } \hfill
  \parbox[b]{0.3\textwidth}{
    \begin{tikzpicture}[->,auto,node distance=2cm, semithick]
      \node (a) at (0,0.5) {$a$};
      \node (b) at (2,0.5) {$b$};

      \path (a) edge [bend left=20] node [above] {$b$} (b);
      \path (b) edge [bend left=20] node [below] {$a$}  (a);

      \node at (1,-1) {$\Acceptor$};
    \end{tikzpicture}
  } \hfill
  \parbox[b]{0.3\textwidth}{    
    \begin{tikzpicture}[->,auto,node distance=2cm, semithick]
      \node [red] (a_1) at (0,0.7) {$(a,\id)$};
      \node [red] (b_1) at (2.5,0.7) {$(b,\id)$};
      \node       (a_a) at (0,3) {$(a,a)$};
      \node       (b_b) at (2.5,3) {$(b,b)$};

      \path (a_1) edge [bend left=20,red] node [above] {$(b,\id)$} (b_1);
      \path (b_1) edge [bend left=20,red] node [below] {$(a,\id)$} (a_1);
      \path (a_a) edge [bend left=20,red] node [above] {$(b,b)$} (b_b);
      \path (b_b) edge [bend left=20,red] node [below] {$(a,a)$} (a_a);

      \node at (1,-1) {$\PSeqAcceptor$\qquad
                       \color{red}{$\widetilde{\PSeqAcceptor} \setminus \PSeqAcceptor$}};
    \end{tikzpicture}
  } \hfill {}
  \caption{Lattice of simple elements and digraphs $\Gamma$, $\PSeqAcceptor$ and $\widetilde{\PSeqAcceptor}$ for the monoid $M=\langle a,b \st a^2=b^2 \rangle^+$.  Vertices and edges in red are contained in $\widetilde{\PSeqAcceptor} \smallsetminus \PSeqAcceptor$.  (That is, the digraph~$\PSeqAcceptor$ consists of the vertices $(a,a)$, and $(b,b)$ without any edges.)}
  \label{F:Perron-Frobenius-aa=bb}
\end{figure}

\begin{figure}
  \hfill
  \parbox[b]{0.3\textwidth}{
    \begin{tikzpicture}[->,auto,node distance=2cm, semithick]
      \node (1)  at (1,0) {$\id$};
      \node (a)  at (0,1) {$a$};
      \node (b)  at (2,1) {$b$};
      \node (ab) at (0,2) {$ab$};
      \node (ba) at (2,2) {$ba$};
      \node (D)  at (1,3) {$\Delta$};

      \path (1)  edge node               {$a$} (a);
      \path (1)  edge node [below right] {$b$} (b);
      \path (a)  edge node               {$b$} (ab);
      \path (b)  edge node               {$a$} (ba);
      \path (ab) edge node               {$a$} (D);
      \path (ba) edge node [above right] {$b$} (D);
      
      \node at (1,-1) {Hasse diagram};
    \end{tikzpicture}
  } \hfill
  \parbox[b]{0.3\textwidth}{
    \begin{tikzpicture}[->,auto,node distance=2cm, semithick]
      \node (a)  at (0,1) {$a$};
      \node (ba) at (2,1) {$ba$};
      \node (b)  at (2,0) {$b$};
      \node (ab) at (0,0) {$ab$};

      \path (a)  edge [loop above,left]  (a);
      \path (b)  edge [loop below,right] (b);
      \path (ba) edge [bend left=10]     (ab);
      \path (ab) edge [bend left=10]     (ba);
      \path (a)  edge                    (ab);
      \path (ab) edge                    (b);
      \path (b)  edge                    (ba);
      \path (ba) edge                    (a);
      
      \node at (1,-1) {$\Acceptor$};
    \end{tikzpicture}
  } \hfill
  \parbox[b]{0.3\textwidth}{
    \begin{tikzpicture}[->,auto,node distance=2cm, semithick]
      \node [red] (a_1)  at (0,2) {$(a,\id)$};
      \node [red] (ba_1) at (2,1) {$(ba,\id)$};
      \node [red] (b_1)  at (2,0) {$(b,\id)$};
      \node [red] (ab_1) at (0,1) {$(ab,\id)$};
      \node (a_b)  at (2,2) {$(a,b)$};
      \node (b_a)  at (0,0) {$(b,a)$};

      \path (a_1)  edge [loop above,left,red]  (a_1);
      \path (b_1)  edge [loop below,right,red] (b_1);
      \path (ba_1) edge [bend left=10,red]     (ab_1);
      \path (ab_1) edge [bend left=10,red]     (ba_1);
      \path (a_1)  edge [red]                  (ab_1);
      \path (ab_1) edge [red]                  (b_1);
      \path (b_1)  edge [red]                  (ba_1);
      \path (ba_1) edge [red]                  (a_1);
      \path (ab_1) edge [red]                  (b_a);
      \path (b_1)  edge [red]                  (b_a);
      \path (ba_1) edge [red]                  (a_b);
      \path (a_1)  edge [red]                  (a_b);
      
      \node (ab_a) at (0,3.5) {$(ab,a)$};
      \node (b_ab) at (2,3.5) {$(b,ab)$};
      \node (a_ba) at (0,4.5) {$(a,ba)$};
      \node (ba_b) at (2,4.5) {$(ba,b)$};

      \path (ba_b) edge [bend left=10,red]     (ab_a);
      \path (ab_a) edge [bend left=10,red]     (ba_b);
      \path (ab_a) edge [red]                   (b_ab);
      \path (ba_b) edge [red]                   (a_ba);
      \path (b_ab) edge [red]                   (ba_b);
      \path (a_ba) edge [red]                   (ab_a);
      \path (a_ba) edge [loop above,left,red]  (a_ba);
      \path (b_ab) edge [loop below,right,red] (b_ab);

      \node at (1,-1) {$\PSeqAcceptor$\qquad
                       \color{red}{$\widetilde{\PSeqAcceptor} \setminus \PSeqAcceptor$}};
    \end{tikzpicture}
  } \hfill {}
  \caption{Lattice of simple elements and digraphs $\Gamma$, $\PSeqAcceptor$ and $\widetilde{\PSeqAcceptor}$ for the monoid $M=\langle a,b \st aba = bab \rangle^+$.  Vertices and edges in red are contained in $\widetilde{\PSeqAcceptor} \smallsetminus \PSeqAcceptor$.  (That is, the digraph~$\PSeqAcceptor$ consists of the vertices $(a,b)$, $(b,a)$, $(b,ab)$, $(a,ba)$, $(ab,a)$ and $(ba,b)$ without any edges.)  To reduce clutter, the edge labels in the acceptor graphs were suppressed.}
  \label{F:Perron-Frobenius-aba=bab}
\end{figure}

Now define a vector $y = (y_t)_{t\in\mathcal{P}}$ by setting $y_{(s,m)}=x_s$ for $(s,m)\in\mathcal{P}$ and a vector $\widetilde{y} = (\widetilde{y}_t)_{t\in\widetilde{\mathcal{P}}}$ by setting $\widetilde{y}_{(s,m)}=x_s$ for $(s,m)\in\widetilde{\mathcal{P}}$.

Clearly, $\widetilde{y}$ is an eigenvector to the eigenvalue $\beta$ of $\widetilde{\PSeqAcceptor}$:
\[
   (\widetilde{\PSeqAcceptor}\widetilde{y})_{(s_1,m_1)}
     = \!\!\!\!
   \sum_{\substack{
               (s_2,m_2)\in\widetilde{\mathcal{P}} \\[0.5ex]
               \partial s_2\wedge s_1 = \id \\[0.5ex]
               m_2=\partial s_2\wedge s_1 m_1
             }} \!\!\!\! \widetilde{y}_{(s_2,m_2)}
     = \!\!
   \sum_{\substack{
               s_2\in\pSimples \\[0.5ex]
               \partial s_2\wedge s_1 = \id
             }} \!\!\!\! x_{s_2}
     =
   (\Acceptor x)_{s_1}
     =
   \beta \cdot x_{s_1}
     =
   \beta \cdot \widetilde{y}_{(s_1,m_1)} .
\]

So it remains to show that for $t\in\mathcal{P}\subseteq\widetilde{\mathcal{P}}$ one has $y_t > 0$ and
$(\PSeqAcceptor y)_t < (\widetilde{\PSeqAcceptor}\widetilde{y})_t$.%
\medskip

\begin{claim*}
For $(s,m)\in\mathcal{P}\subseteq\widetilde{\mathcal{P}}$ one has $y_{(s,m)} > 0$.
\end{claim*}
\noindent
As every $s\in\pSimples$ is essential and $\calL$ is essentially transitive,
for any $s_1,s_2\in\pSimples$ there exists a positive integer $\ell$ such that $(\Gamma^\ell)_{s_1,s_2}>0$, whence $x_{s_2}>0$ implies $(\Gamma^\ell x)_{s_1}>0$, as $\Gamma$ and $x$ are non-negative.  As $x$ is an eigenvector of $\Gamma$ and $x_{s_2}>0$ must hold for at least one $s_2\in\pSimples$, we have $x_{s_1}>0$ for every $s_1\in\pSimples$ and thus $y_{(s,m)}=x_{s}>0$ for every $(s,m)\in\mathcal{P}$, showing the claim.
\medskip

\begin{claim*}
For $(s,m)\in\mathcal{P}\subseteq\widetilde{\mathcal{P}}$ one has
$(\PSeqAcceptor y)_{(s_1,m_1)} < (\widetilde{\PSeqAcceptor}\widetilde{y})_{(s_1,m_1)}$.
\end{claim*}
\noindent
We obtain $\PSeqAcceptor$ from $\widetilde{\PSeqAcceptor}$ by
\begin{inparaenum}[1.)]
 \item removing all edges ending in $(s_1,m_1)$ if $m_1=\id$ or $s_1m_1=\Delta$; and
 \item removing the edge $(s_2,m_2)\to (s_1,m_1)$ if $m_2=\id$.
\end{inparaenum}

So it is sufficient to show that for all $(s_1,m_1)\in\mathcal{P}$ such that $s_1m_1\ne\Delta$, there exists $s_2\in\pSimples$ such that $\partial s_2\wedge s_1m_1=\id$ and $y_{(s_2,m_2)}=x_{s_2}>0$.
The latter holds as all proper simple elements are essential, and thus $s_1m_1\ne\Delta$ implies the existence of an essential $s_2$ such that $s_2|s_1m_1$, and since $x_{s_2}>0$ holds.
\end{proof}

\begin{remark}
\autoref{T:BoundedExpectedPD} shows that the hypotheses of \cite[Theorem~4.8]{GT13} cannot be satisfied.
\end{remark}

\begin{corollary}\label{C:BoundedExpectedPD}
Assume that every proper simple element of $M$ is essential and that the language $\calL_M$ of normal forms is essentially transitive, let $\nu_{k}$ be the uniform probability measure on $\calL_M^{(k)}$, and let $\mu_\Atoms$ be the uniform probability distribution on the set $\Atoms$ of atoms of $M$.

The expected value $\Expect_{\nu_{k} \times \mu_\Atoms}[\pd]$ of the penetration distance with respect to $\nu_k\times \mu_\Atoms$ is uniformly bounded (that is, bounded independently of $k$).
\end{corollary}

\begin{proof}
The claim follows with \autoref{T:BoundedExpectedPD} and \cite[Theorem 4.7]{GT13}.
\end{proof}

\smallskip\noindent
We now turn to the analysis of growth rates of \ZS{} products with respect to a \ZS{} Garside structure.

\begin{lemma}\label{L:Sums}
For $c>1$ and $m\in\NN$ the following hold:
\begin{enumerate} \itemsep 0em \vspace{-0.5\topskip}
\item One has $\sum_{j=0}^{k-1} j^m c^j \in \Theta(k^m c^k)$.
\item One has $\sum_{j=0}^{k-1} j^m \in \Theta(k^{m+1})$.
\end{enumerate}
\end{lemma}

\begin{proof}
The second claim holds by Bernoulli's formula~\cite[p.\,283]{concrete}.  For the first claim, observe that for any $k\ge2$ one has
\[
\frac1{2^m c}
 \le \frac{1}{k^m c^k} (k-1)^m c^{k-1}
 \le \frac{1}{k^m c^k} \sum_{j=0}^{k-1} j^m c^j
 <   \sum_{j=1}^k c^{-j}
 <   \frac{1}{c-1}  \;.
\]
\end{proof}

\begin{samepage}
\begin{lemma}\label{L:GrowthRates}
Using \autoref{N:GrowthRates}, the following hold:
\begin{enumerate} \itemsep 0em \vspace{-0.5\topskip}
\item One has $\beta_M=0$ if and only if $\gamma_M=1$ and $r_M=0$ hold.
\item One has $\beta_M=1$ if and only if $\gamma_M=1$ and $r_M\ge1$ hold.
      Moreover, in this case, one has $q_M=r_M-1$.
\item One has $\beta_M>1$ if and only if $\gamma_M>1$ holds.
      Moreover, in this case, one has $\beta_M=\gamma_M$ and $q_M=r_M$.
\end{enumerate}
\end{lemma}
\end{samepage}
\begin{proof}
As $\calLbar^{(k)} = \bigsqcup_{j=0}^k \Delta^{k-j} \calL^{(j)}$ holds, one has
$\Big|\calLbar^{(k)}\Big| = \sum_{j=0}^k \big|\calL^{(j)}\big|$.
Firstly observe that $\beta_M=0$ holds if and only if one has $\big|\calL^{(k)}\big|=0$ for sufficiently large $k$.  The latter happens if and only if $\Big|\calLbar^{(k)}\Big|$ is eventually constant, which is equivalent to $\gamma_M=1$ and $r_M=0$.  So the first claim is shown.

Using \autoref{N:GrowthRates} and \autoref{L:Sums}, we obtain
\[
  k^{r_M} {\gamma_M}^k \in \Theta\Bigg( \sum_{j=0}^k j^{q_M} {\beta_M}^j \Bigg)
    = \begin{cases}
       \Theta\big( k^{q_M} {\beta_M}^k \big) & \text{if $\beta_M>1$} \\[1.0ex]
       \Theta\big( k^{q_M+1} \big)         & \text{if $\beta_M=1$} \;,
      \end{cases}
\]
which implies the remaining claims.
\end{proof}

\begin{proposition}\label{P:GrowthRatesProduct}
Assume that $(M,G,H,\Delta_M, \Delta_G, \Delta_H)$ is a \ZS{} Garside structure, let $\beta_M,\gamma_M$ and $q_M,r_M$ be as in~\autoref{N:GrowthRates}, and let $\beta_G,\gamma_G,\beta_H,\gamma_H\in\{0\}\cup[1,\infty[$ and $q_G,r_G,q_H,r_H\in\NN$ be the corresponding constants for~$G$ respectively~$H$.

The following table gives $\beta_M,\gamma_M,q_M,r_M$ in terms of $\beta_G,\gamma_G,\beta_H,\gamma_H$ and $q_G,r_G,q_H,r_H$:
\vspace*{-2ex}
\[
\begin{array}{c||c|c|c}
                  & \beta_H=0          & \beta_H=1          & \beta_H>1                  \\
                  & \gamma_H=1         & \gamma_H=1         & \gamma_H=\beta_H           \\
                  & r_H=0              & r_H=q_H+1          & r_H=q_H
\rule[-1.2ex]{0pt}{2.5ex}\\ \hline\hline\rule{0pt}{2.5ex}
\beta_G=0         & \beta_M=1          & \beta_M=1          & \beta_M=\beta_H            \\
\gamma_G=1        & q_M=0              & q_M=q_H+1          & q_M=q_H+1                  \\
r_G=0             & \gamma_M=1         & \gamma_M=1         & \gamma_M=\gamma_H          \\
                  & r_M=1              & r_M=r_H+1          & r_M=r_H+1
\rule[-1.2ex]{0pt}{2.5ex}\\ \hline\rule{0pt}{2.5ex}
\beta_G=1         & \beta_M=1          & \beta_M=1          & \beta_M=\beta_H            \\
\gamma_G=1        & q_M=q_G+1          & q_M=q_G+q_H+2      & q_M=q_G+q_H+2              \\
r_G=q_G+1         & \gamma_M=1         & \gamma_M=1         & \gamma_M=\gamma_H          \\
                  & r_M=r_G+1          & r_M=r_G+r_H+1      & r_M=r_G+r_H+1
\rule[-1.2ex]{0pt}{2.5ex}\\ \hline\rule{0pt}{2.5ex}
\beta_G>1         & \beta_M=\beta_G    & \beta_M=\beta_G    & \beta_M=\beta_G\beta_H     \\
\gamma_G=\beta_G  & q_M=q_G+1          & q_M=q_G+q_H+2      & q_M=q_G+q_H                \\
r_G=q_G           & \gamma_M=\gamma_G  & \gamma_M=\gamma_G  & \gamma_M=\gamma_G\gamma_H  \\
                  & r_M=r_G+1          & r_M=r_G+r_H+1      & r_M=r_G+r_H
\end{array}
\]
\end{proposition}
\begin{proof}
First note that, by \autoref{L:GrowthRates}, the cases in the table are correct and exhaustive.

For $k\in\NN$ consider any $x\in M$ with $\NF(x)\in\calLbar_M^{(k)}$.  By \autoref{T:ZappaSzep}, there exist unique elements $g_x\in G$ and $h_x\in H$ such that $x=g_x\join h_x$, and moreover, $\NF(g_x)\in \calLbar_G^{(k_G)}$ and $\NF(h_x)\in \calLbar_H^{(k_H)}$ with $k=\max\{k_G,k_H\}$.

Indeed, the map $$\NF(x)\mapsto \big(\NF(g_x),\NF(h_x)\big)$$ is a bijection from $\calLbar_M$ to $\calLbar_G \times \calLbar_H$ by \autoref{T:ZappaSzep}, whence one has
%
%
\[
  \big|\calLbar_M^{(k)}\big| = \big|\calLbar_G^{(k)}\big| \cdot \big|\calLbar_H^{(k)}\big|
      + \sum_{j=0}^{k-1} \big|\calLbar_G^{(k)}\big| \cdot \big|\calLbar_H^{(j)}\big|
      + \sum_{j=0}^{k-1} \big|\calLbar_G^{(j)}\big| \cdot \big|\calLbar_H^{(k)}\big|
\]
and thus
\begin{align*}
  k^{r_M} {\gamma_M}^k  \in \Theta\Bigg(
    k^{r_G} {\gamma_G}^k k^{r_H} {\gamma_H}^k
    + k^{r_G} {\gamma_G}^k & \bigg(\sum_{j=0}^{k-1} j^{r_H} {\gamma_H}^j\bigg) \\
    & + \bigg(\sum_{j=0}^{k-1} j^{r_G} {\gamma_G}^j\bigg) k^{r_H} {\gamma_H}^k
  \Bigg)  \;.
\end{align*}
The claimed equalities for $\gamma_M$ and $r_M$ are easily verified using \autoref{L:Sums}, and the claimed equalities for $\beta_M$ and $q_M$ then follow with \autoref{L:GrowthRates}.
\end{proof}

\begin{corollary}\label{C:GrowthRatesProduct}
Assume that $(M,G,H,\Delta_M, \Delta_G, \Delta_H)$ is a \ZS{} Garside structure, and let $\beta_G$ and $\beta_H$ be the exponential growth rates of the regular languages~$\calL_G$ respectively~$\calL_H$.

Then $\big|\calL_M^{(k)}\big| \in \Theta\Big( \big|\calL_G^{(k)}\big| \cdot \big|\calL_H^{(k)}\big| \Big)$ holds if and only if $\beta_G>1$ and $\beta_H>1$.
\end{corollary}
\begin{proof}
Using \autoref{N:GrowthRates}, $\big|\calL_M^{(k)}\big| \in \Theta\Big( \big|\calL_G^{(k)}\big| \cdot \big|\calL_H^{(k)}\big| \Big)$ is equivalent to $\beta_M=\beta_G\beta_H$ and $q_M=q_G+q_H$.
The claim then follows with \autoref{P:GrowthRatesProduct}.
\end{proof}

\begin{notation}\label{N:PSeqProduct}
Assume that $(M,G,H,\Delta_M, \Delta_G, \Delta_H)$ is a \ZS{} Garside structure.  Given $g=g_1|\cdots| g_k\in \calL_G^{(k)}$ and $h=h_1|\cdots |h_k\in \calL_H^{(k)}$,
consider the normal form $g'_1 h'_1|\cdots| g'_k h'_k \in \calL_M^{(k)}$ of
$g_1\cdots g_k \join h_1\cdots h_k$,
and define
\[
   \pi_{g,h} := (g'_1 h'_1, m_1)\ldot\cdots\ldot (g'_k h'_k, m_k)
   \;,
\]
where $m_k=\partial_H(h'_k)$ (that is, $h'_k m_k=\Delta_H$) and $m_i = \partial_M(g'_i h'_i) \wedge_M g'_{i+1} h'_{i+1} m_{i+1}$ for $i=1,\ldots,k-1$.
\end{notation}

\begin{lemma}\label{L:PSeqProduct}
In the situation of \autoref{N:PSeqProduct}, one has $\pi_{g,h} \in \PSeq_M^{(k)}$.
\end{lemma}
\begin{proof}
Using \cite[Lemma~42]{Zappa-Szep}, we have $\Delta_H \prefix g'_i h'_i m_i \neq \Delta_M$ and $m_i\ne \id$ for all $i$ by induction, so $\pi_{g,h}\in\PSeq_M^{(k)}$.
\end{proof}

\begin{proposition}\label{P:PSeqProduct}
Assume that $(M,G,H,\Delta_M, \Delta_G, \Delta_H)$ is a \ZS{} Garside structure, let $\alpha_M,\beta_M$ and $p_M,q_M$ be as in~\autoref{N:GrowthRates}, and let $\alpha_G,\beta_G,\alpha_H,\beta_H\in\{0\}\cup[1,\infty[$ and $p_G,q_G,p_H,q_H\in\NN$ the corresponding constants for~$G$ respectively~$H$.

Then one has the following:
\begin{enumerate} \itemsep 0em \vspace{-0.5\topskip}
\item If one has $\beta_G,\beta_H>0$, then $\alpha_M=\beta_M$ holds.
\item If $\beta_G,\beta_H>1$, then $\alpha_M=\beta_M$ and $p_M=q_M$ hold.
\end{enumerate}
\end{proposition}
\begin{proof}
By \autoref{L:PSeqProduct}, we have a map $\calL_G^{(k)} \times \calL_H^{(k)} \to \PSeq_M^{(k)}$ given by the assignment $(g,h)\mapsto \pi_{g,h}$.
This map is injective by \autoref{T:ZappaSzep}, so we have
$\big|\PSeq_M^{(k)}\big| \ge \big|\calL_M^{(k)}\big| \cdot \big|\calL_N^{(k)}\big|$.
Thus $k^{q_G+q_H} (\beta_G\beta_H)^k \in O(k^{p_M} {\alpha_M}^k) \subseteq O(k^{q_M} {\beta_M}^k)$, where the final inclusion holds by~\cite[Corollary~4.4]{GT13}.

If $\beta_G,\beta_H>0$, then we have $\beta_M=\beta_G\beta_H$ by \autoref{P:GrowthRatesProduct}, and thus obtain $\alpha_M=\beta_M$.
Similarly, if $\beta_G,\beta_H>1$, then we have $\beta_M=\beta_G\beta_H$ and $q_M=q_G+q_H$ by \autoref{P:GrowthRatesProduct}, and thus obtain $\alpha_M=\beta_M$ and $p_M=q_M$.
\end{proof}

\begin{notation}
Assume that $M$ is a Garside monoid with set of proper simple elements $\pSimples$.
For $s\in \pSimples$ and $k\ge 1$, we define
\[
  \calL_M^{(k)}(s) := \calL_M^{(k)} \cap (\pSimples)^*\ldot s
    = \Big\{ s_1\ldot\cdots\ldot s_k\in \calL_M^{(k)} : s_k = s \Big\}
  \;.
\]
\end{notation}

\begin{lemma}[{\cite[Lemma~4.10]{GT13}}]\label{L:RestrictedNF}
If $M$ is a Garside monoid such that all proper simple elements of $M$ are essential and the language $\calL_M$ is essentially transitive, then one has
$\big|\calL_M^{(k)}(s)\big| \in \Theta\Big(\big|\calL_M^{(k)}\big|\Big)$ for all $s\in\pSimples$.
\end{lemma}

\begin{theorem}\label{T:PSeqProduct}
Assume that $(M,G,H,\Delta_M, \Delta_G, \Delta_H)$ is a \ZS{} Garside structure, that all proper simple elements of~$H$ are essential, that the language~$\calL_H$ is essentially transitive, and that the exponential growth rates~$\beta_G$ of~$\calL_G$  and~$\beta_H$ of~$\calL_H$ satisfy $\beta_G,\beta_H>1$.

If $\nu_{k}$ is the uniform probability measure on~$\calL_M^{(k)}$ and~$\mu_\Atoms$ is the uniform probability distribution on the set~$\Atoms$ of atoms of~$M$, then the expected value
$\Expect_{\nu_{k} \times \mu_\Atoms}[\pd]$ diverges, that is, one has $\lim_{k\to\infty} \Expect_{\nu_{k} \times \mu_\Atoms}[\pd] = \infty$.
\end{theorem}
\begin{proof}
For any atom $a\in\Atoms\cap H$, any $g\in\calL_G^{(k)}$ and any $h\in\calL_H^{(k)}(\rightcomplement_H a)$, 
the sequence $\pi_{g,h}$ defined in \autoref{N:PSeqProduct} is a penetration sequence establishing $\pd(x_{(g,h)},a)\ge k$ for some $x_{(g,h)}\in M$.  Moreover, the map $(g,h)\mapsto x_{(g,h)}$ is injective.

Using \autoref{L:RestrictedNF} and \autoref{C:GrowthRatesProduct}, we have
\begin{align*}
 \Expect_{\nu_{k} \times \mu_\Atoms}[\pd]
   & \ge \sum_{a\in\Atoms\cap H} k
                 \frac{\big|\calL_G^{(k)}\big|\cdot \big|\calL_H^{(k)}(\rightcomplement_H a)\big|}
                      {\big|\calL_M^{(k)}\big|\cdot \big|\Atoms\big|}
     \in \Theta(k)
   \;,
\end{align*}
proving $\lim_{k\to\infty}\Expect_{\nu_{k} \times \mu_\Atoms}[\pd] = \infty$ as claimed.
\end{proof}

\begin{example}\label{E:UnboundedExpectedPD}
We see in particular that essential transitivity is necessary for the statement of \autoref{T:BoundedExpectedPD}:

Consider $M=G\times H$, where $G=H=\Artin{A}_2$.
The lattice of simple elements of $G=H$ and the acceptor~$\Gamma$ of $\calL_G=\calL_H$ are shown in \autoref{F:UnboundedExpectedPD}.
One sees that all proper simple elements of~$G$, respectively of~$H$, and thus of~$M$, are essential and the languages~$\calL_G$ and~$\calL_H$ are essentially transitive, and it is easy to check that $\beta_G=\beta_H=2$.
Hence, by \autoref{T:PSeqProduct},~$M$ has unbounded expected penetration distance.

Every proper simple element of~$M$ is essential, so~$M$ satisfies all the hypotheses of \autoref{T:BoundedExpectedPD}, except for essential transitivity.

  \begin{figure}
  \hfill
  \parbox[c]{0.45\textwidth}{
    \begin{xy}
      0;<3em,0em>:<0em,3em>::
      (1,-0.6)*+{\text{Hasse diagram}};
      (1,0)*+{\id}="e";
      (0,1)*+{a}="a";
      (2,1)*+{b}="b";
      (0,2)*+{ab}="ab";
      (2,2)*+{ba}="ba";
      (1,3)*+{\Delta}="Delta";
      {\ar@{->}^{a} "e";"a"};
      {\ar@{->}_{b} "e";"b"};
      {\ar@{->}^{b} "a";"ab"};
      {\ar@{->}_{a} "b";"ba"};
      {\ar@{->}^{a} "ab";"Delta"};
      {\ar@{->}_{b} "ba";"Delta"};
    \end{xy}
  } \hfill
  \parbox[c]{0.45\textwidth}{
    \begin{xy}
      0;<3em,0em>:<0em,3em>::
      (1,-0.75)*+{\Acceptor};
      (0,2)*+{a}="a";
      (2,0)*+{b}="b";
      (0,0)*+{ab}="ab";
      (2,2)*+{ba}="ba";
      {\ar@(u,l)_{a} "a";"a"};
      {\ar@{->}_{ab} "a";"ab"};
      {\ar@{->}_{b} "ab";"b"};
      {\ar@/_/_{ba} "ab";"ba"};
      {\ar@(d,r)_{b} "b";"b"};
      {\ar@{->}_{ba} "b";"ba"};
      {\ar@{->}_{a} "ba";"a"};
      {\ar@/_/_{ab} "ba";"ab"};
    \end{xy}
  } \hfill {}
  \caption{The lattice of simple elements and the acceptor for the regular language of words in normal form for the monoid from \autoref{E:UnboundedExpectedPD}.}
  \label{F:UnboundedExpectedPD}
  \end{figure}
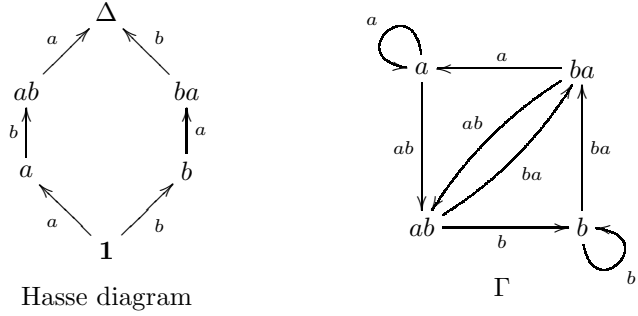
\end{example}

\comment{
\begin{remark}
Consider the monoid $G = \langle a, b \st aba = b^2\rangle^+$; it is essentially transitive, but there are proper simple elements that are not essential.

The statement $\alpha<\beta$ holds, but two things go wrong in the proof of \autoref{T:BoundedExpectedPD}:
\begin{enumerate} \itemsep 0em \vspace{-0.5\topskip}
\item
There are vertices $(s,m)\in\mathcal{P}$ for which $(\PSeqAcceptor y)_{(s,m)}=(\widetilde{\PSeqAcceptor}y)_{(s,m)}$, namely $(ab,a)$, $(a,ba)$ and $(b,b)$.
Specifically, we have
\begin{align*}
 &(\PSeqAcceptor y)_{(b,b)} =  y_{(ab,ab)} + y_{(bab,a)} = (\widetilde{\PSeqAcceptor}y)_{(b,b)} \;,\\
 &(\PSeqAcceptor y)_{(ab,a)}=y_{(a,b)}+y_{(b^2,b)}+y_{(ba,b)}=(\widetilde{\PSeqAcceptor}y)_{(ab,a)} \;,
  \text{ and} \\
 &(\PSeqAcceptor y)_{(a,ba)}=y_{(a,b)}+y_{(b^2,b)}+y_{(ba,b)}=(\widetilde{\PSeqAcceptor}y)_{(a,ba)} \;.
\end{align*}
In this particular example, each offending vertex $(s,m)$ of this type has out-degree~$0$, whence one can increase $y_{(s,m)}$ without changing $(\widetilde{\PSeqAcceptor}y)_{(s,m)}$ or $(\PSeqAcceptor y)_{(s,m)}$, reducing the ratio $\frac{(\PSeqAcceptor y)_{(s,m)}}{y_{(s,m)}}$.

Unfortunately, this doesn't hold in general: in \texttt{ToyGarsideGroup2}, we have vertices of this kind that have positive out-degree.

\item
There are vertices $(s,m)$ for which $y_{(s,m)}=0$; in particular, this is necessarily the case for all vertices of in-degree~$0$.  Specifically, we have $y_{(b^2,1)}=y_{(b^2,b)}=0$.

In this particular example, each offending vertex of this type in $\mathcal{P}$ has only edges ending in a vertex of out-degree~0, so this part of the problem can also be solved by modifying $y$.

Unfortunately, this doesn't hold in general either: in \texttt{ToyGarsideGroup2}, we have vertices of this kind (also of in-degree~0) that have edges ending in vertices of positive out-degree.
\end{enumerate}

It would be nice if we could either find an example showing that the requirement that all proper simple elements are essential is needed, or patch the proof to deal with non-essential proper simple elements if it isn't.  I have absolutely no idea which of these\dots
\end{remark}
}

\section{Artin monoids}\label{S:ArtinMonoids}

The aim of this section is to determine the essential simple elements
of Artin monoids of spherical type and to determine when these monoids
are essentially transitive.

In \cite[Lemma~3.4]{Caruso13}, Caruso shows that, in our terminology,
the language of normal forms of an Artin monoid of type
$\Artin{A}$ is essentially 5-transitive.  In Lemmata
\ref{connecting-atoms} to \ref{rev-u} we will generalise Caruso's
construction and reproduce this result in \autoref{type-A-ess-trans}.
We then go on to extend this result to all irreducible Artin monoids of spherical type.

For the rest of this section we will assume that $M$ is an Artin monoid of spherical type with set of atoms~$\Atoms$.

\begin{proposition}\label{Artin-ProperSimplesEssential}
  Every proper simple element of~$M$ is essential, i.e.\ $\Ess = \pSimples$.
\end{proposition}
\begin{proof}
  There are no proper simple elements in $\Artin{A}_1$, so we can assume that~$|\Atoms|>1$.
  Given any $x \in \pSimples$ pick an atom $a \in \Fin(x)$ and let
  $x_i = \rightJoin \Start(x_{i-1})$ for $i=1,2,\ldots$, where $x_0 = x$.
  Then we have $a\in\pSimples$ and
  \[ \cdots | x_2 | x_1 | x | a | a | \cdots. \]
  Recall that in any Artin monoid of spherical type, the Garside element is the ($\prefix$- and~$\suffix$-) least common upper bound of the set of atoms.
  We obtain $x_i\in\pSimples$ by induction:
  If $x_{i-1}\in\pSimples$, then $\Start(x_{i-1})$ and $\Atoms\setminus\Start(x_{i-1})$ are both non-empty.
  Thus, $x_i$ is the Garside element of a non-trivial proper parabolic submonoid of~$M$; in particular, $x_i\in\pSimples$.
%
\end{proof}

\begin{proposition} \label{reducible-non-transitive}
  If $M$ is reducible then it is not essentially transitive.
\end{proposition}
\begin{proof}
  If $M$ is reducible then $M = M_1 \times M_2 \times \cdots \times
  M_k$ where the $M_i$ are irreducible.  Hence the claim follows by
  \autoref{product-not-transitive}.
\end{proof}

\begin{lemma} \label{connecting-atoms}
  Suppose that $a$ and $b$ are two atoms which lie in the same
  connected component of the Coxeter graph of~$M$.  Then there exist
  simple elements~$x$ and~$y$ such that $\Start(x) = \{ a \}$,
  $\Fin(x) = \{ b \}$, $\Start(y) = \Atoms\setminus\{a\}$ and $\Fin(y)
  = \Atoms\setminus\{b\}$.
\end{lemma}
\begin{proof}
  Suppose that we have an embedded path $a=a_1
  \overset{i_1}{\text{ --- }} a_2 \overset{i_2}{\text{ --- }} \cdots
  \overset{i_{k-1}}{\text{ --- }} a_k=b$ in the Coxeter graph of $M$.
  Let $x = a_1 a_2 \cdots a_k$.  There are no subwords which match any
  of the relations, hence $\Start(x) = \{ a \}$ and $\Fin(x) = \{ b
  \}$.  Furthermore, the only way $x$ can be written as a product is
  as $x = (a_1 \cdots a_p)(a_{p+1} \cdots a_k)$, which can never be in
  normal form.  Hence~$x$ has canonical length at most one, that is,~$x$ is a simple
  element.

  Let $y = \complement (a_k \cdots a_2 a_1)$, then $\Start(y) = \Atoms
  \setminus \Fin(a_k \cdots a_2 a_1) = \Atoms \setminus \{a\}$ and,
  similarly, $\Fin(y) = \Atoms \setminus \Start(a_k \cdots a_2 a_1) =
  \Atoms \setminus \{b\}$.  
\end{proof}

Suppose that $k\ge 2$ and that $1 \text{ --- } 2 \text{ ---} \cdots \text{--- } (k-1)$ is
a subgraph of the Coxeter graph of~$M$, so we have a parabolic submonoid of type $\Artin{A}_{k-1}$.  In this situation we
have a map from $\Artin{A}_{k-1}$ to the symmetric group on the set
$\{ 1, 2, \ldots, k \}$ given by mapping each atom $i$ to the
transposition $(i, i+1)$.  This map is a bijection when we restrict to
the set of simple elements of this submonoid.

Suppose that $a \in \Atoms_{\Artin{A}_{k-1}}$ and $x \in
\Simples_{\Artin{A}_{k-1}}$ are an atom and a simple element of this
submonoid.  Let $\pi$ be the permutation induced by $x$.  Then we have
that $a \in \Start(x)$ if and only if $\pi(a+1) < \pi(a)$, and $a \in
\Fin(X)$ if and only if $\pi^{-1}(a+1) < \pi^{-1}(a)$ \cite[Proposition~1.5.3]{CombCox}.

\begin{lemma}[\cite{Caruso13}]\label{one-to-many}
  Suppose that $1 \text{ --- } 2 \text{ ---} \cdots \text{--- } (k-1)$, where $k\ge 2$, is
  a subgraph of the Coxeter graph of $M$. Let $u = u(1,2,\ldots,k-1)$ be the
  braid which corresponds to the permutation
  \[ \pi_u = \left(\begin{array}{cccccccc}
      1 & 
      2 & 
      \cdots & 
      \left\lfloor \frac{k}{2} \right\rfloor &
      \left\lfloor \frac{k}{2} \right\rfloor + 1 &
      \left\lfloor \frac{k}{2} \right\rfloor + 2 &
      \cdots & 
      k \\
      2 &
      4 &
      \cdots & 
      2\left\lfloor \frac{k}{2} \right\rfloor &
      1 &
      3 &
      \cdots & 
      2\left\lceil \frac{k}{2} \right\rceil - 1
    \end{array}\right)
    \;.
  \]
  Then one has
  \[ 
  \Start(u) = \left\{\left\lfloor\frac{k}{2}\right\rfloor\right\}
  \qquad
  \Fin(u) = \left\{1,3,\ldots,2\left\lfloor\frac{k}{2}\right\rfloor-1\right\}.
  \]
\end{lemma}
\begin{proof}
  It is clear that the only atom $a \in
  \Atoms_{\Artin{A}_{k-1}}$ for which $\pi_u(a+1) < \pi_u(a)$ is~$a=\left\lfloor\frac{k}{2}\right\rfloor$, hence $\Start(u) = \left\{\left\lfloor\frac{k}{2}\right\rfloor\right\}$.  Similarly,
  $\pi_u^{-1}(a+1) < \pi_u^{-1}(a)$ if and only if~$a$ is odd, hence
  $\Fin(u)$ consists of all the odd atoms.
\end{proof}

\begin{lemma}[\cite{Caruso13}]\label{rev-u}
  Suppose that $1 \text{ --- } 2 \text{ ---} \cdots \text{--- } (k-1)$, where $k\ge 2$,
  is a subgraph of the Coxeter graph of $M$, let $u$ be the element defined in \autoref{one-to-many},
  and let 
  \[ 
  v = v(1,2,\ldots,k-1) = (\rev u) \cdot D
  \;,
  \]
  where $D = \Join \left\{ a \in \Atoms_M : a \ne \left\lfloor\tfrac{k}{2}\right\rfloor\right\}$
  and $\rev u$ is the simple element obtained by reversing any expression of~$u$ as a product of atoms.
  
  Then $v$ is a simple element and its finishing set
  contains every atom except possibly~$\left\lfloor\frac{k}{2}\right\rfloor$,
  that is, one has $\Fin(v) \supseteq \Atoms_M \setminus
  \left\{\left\lfloor\frac{k}{2}\right\rfloor\right\}$.
\end{lemma}
\begin{proof}
  As $\Start(u) = \left\{\left\lfloor\frac{k}{2}\right\rfloor\right\}$
  we have that $\Fin(\rev u) =
  \left\{\left\lfloor\frac{k}{2}\right\rfloor\right\}$ hence
  $\Start(\complement\rev u) =
  \Atoms\setminus\left\{\left\lfloor\frac{k}{2}\right\rfloor\right\}$.  Therefore
  $\Join\left(\Atoms\setminus\{\left\lfloor\frac{k}{2}\right\rfloor\}\right)
  \prefix \complement\rev u$ and so $v$ is simple.

  The element $D$ is a Garside element of a parabolic submonoid of~$M$, thus it is balanced, whence one has
  $\Atoms_M \setminus \left\{\left\lfloor\frac{k}{2}\right\rfloor\right\} \subseteq \Fin(D) \subseteq \Fin(v)$.
\end{proof}

\begin{proposition}[\cite{Caruso13}] \label{type-A-ess-trans}
  Suppose that $M$ is the Artin monoid of type $\Artin{A}_{n-1}$, where $n-1\ge 2$.
  \[
    \begin{xy}
      0;<2em,0em>:<0em,2em>::
      (1,0)*+{1}="1";
      (2,0)*+{2}="2";
      (3,0)*+{3}="3";
      (3.75,0)="4";
      (4.6,0)="5";
      (6,0)*+{(n-1)}="6";
      {\ar@{-}     "1";"2"};
      {\ar@{-}     "2";"3"};
      {\ar@{-}     "3";"4"};
      {\ar@{..}    "4";"5"};
      {\ar@{-}     "5";"6"};
    \end{xy}
  \]
  Then $M$ is essentially 5-transitive.
\end{proposition}
\begin{proof}
  By \autoref{Artin-ProperSimplesEssential}, one has $\Ess=\pSimples\neq\emptyset$.
  Suppose $x, y \in \pSimples$.  We will construct elements $x_1, x_2,
  x_3, x_4 \in \pSimples$ that satisfy $x|x_1|x_2|x_3|x_4|y$.

  Suppose that $a$ is an atom in the finishing set of $x$.  By
  \autoref{connecting-atoms}, there exists $x_1$ such that
  $\Start(x_1) = \{ a \}$ and $\Fin(x_1) =
  \left\{\left\lfloor\frac{n}{2}\right\rfloor\right\}$.  For $x_2 = u(1,2,\ldots,n-1)$,
  we have $\Start(x_2) = \left\{\left\lfloor\frac{n}{2}\right\rfloor\right\}$ by \autoref{one-to-many}, and then
  $x|x_1|x_2$ by \autoref{Artin-normal-form}.

  Similarly, suppose that $b$ is an atom not in the starting set of
  $y$.  By \autoref{connecting-atoms}, there exists $x_4$ such that
  $\Fin(x_4) = \Atoms \setminus \{b\}$ and $\Start(x_4) = \Atoms
  \setminus \left\{\left\lfloor\frac{n}{2}\right\rfloor\right\}$.  Let
  $x_3 = v(1,2,\ldots,n-1)$, so by \autoref{rev-u} we have
  $\Fin(x_3)\supseteq\Atoms\setminus\left\{\left\lfloor\frac{n}{2}\right\rfloor\right\}$, whence
  we have $x_3|x_4|y$ by \autoref{Artin-normal-form}.

  It remains to show that $x_2|x_3$, or equivalently that $\Fin(x_2)
  \supseteq \Start(x_3)$.  We have $\Fin(x_2) =
  \left\{1,3,\ldots,2\left\lfloor\frac{n}{2}\right\rfloor-1\right\}$ by \autoref{one-to-many}.

  Consider the permutation induced by $x_3=v$.
  The permutation induced by $\rev u$ takes the set of even numbered strings to
  $\{1,2,\ldots,\lfloor\frac{n}{2}\rfloor\}$ and the set of odd numbered
  strings to $\{\lfloor\frac{n}{2}\rfloor+1,\lfloor\frac{n}{2}\rfloor+2,\ldots,n\}$, and
  $\Join\left\{ a \in \Atoms_M : a \ne \left\lfloor\frac{n}{2}\right\rfloor\right\}$ performs a
  half-twist on both of these subsets, so in particular leaves them invariant.
  Hence, for all $i>0$ we have
  \[ (2i\pm 1)\cdot x_3 > \frac{n}{2} \ge (2i)\cdot x_3\;, \]
  whence $\Start(x_3) = \{ 1, 3, \ldots 2\left\lfloor\frac{n}{2}\right\rfloor - 1\} = \Fin(x_2)$, and so $x_2|x_3$ by \autoref{Artin-normal-form}.
\end{proof}

\begin{remark}
Calculating powers of the adjacency matrices, it can be shown that the acceptor~$\Acceptor_{\Artin{A}_n}$ has a diameter of~5 for $n\in\{4,5,\ldots,11\}$, so the statement of \autoref{type-A-ess-trans} cannot be sharpened in general.
The diameter of~$\Acceptor_{\Artin{A}_2}$ is~2 and the diameter of~$\Acceptor_{\Artin{A}_3}$ is~4.
\end{remark}

\begin{proposition} \label{type-B-ess-trans}
  Suppose that $M$ is the Artin monoid of type $\Artin{B}_n$, where $n\ge 2$.
  \[
    \begin{xy}
      0;<2em,0em>:<0em,2em>::
      (0,0)*+{0}="0";
      (1,0)*+{1}="1";
      (2,0)*+{2}="2";
      (3,0)*+{3}="3";
      (3.75,0)="4";
      (4.6,0)="5";
      (6,0)*+{(n-1)}="6";
      {\ar@{-}^{4}  "0";"1"};
      {\ar@{-}     "1";"2"};
      {\ar@{-}     "2";"3"};
      {\ar@{-}     "3";"4"};
      {\ar@{..}    "4";"5"};
      {\ar@{-}     "5";"6"};
    \end{xy}
  \]
  Then $M$ is essentially 5-transitive.
\end{proposition}
\begin{proof}
  Let $x_2 = u(1,2, \ldots, n-1)$ and $x_3 = v(1,2,\ldots,n-1)$.
  As in the proof of \autoref{type-A-ess-trans}, it suffices to show that $\Start(x_3) \subseteq \{ 1, 3, \ldots 2\left\lfloor\frac{n}{2}\right\rfloor - 1\}$ by \autoref{Artin-normal-form}, \autoref{Artin-ProperSimplesEssential}, \autoref{connecting-atoms}, \autoref{one-to-many}, and \autoref{rev-u}.

  Simple elements in $\Artin{B}_n$ correspond to signed permutations of $\{1,\ldots,n\}$.
  More precisely, the atoms $i = 1, 2, \ldots, n-1$ correspond to the signed
  transpositions $(i,i+1)(-i,-(i+1))$ and the atom~$0$ corresponds to the signed transposition $(1,-1)$.
  The starting set of an element can be computed from its induced signed
  permutation \cite[Proposition 8.1.2]{CombCox}: One has
  \[
    \Start(x_3) = \{ i \in \Atoms : i\cdot x_3 > (i+1)\cdot x_3\}
    \;,
  \]
  where we use the convention $0\cdot x_3 = 0$.

  The signed permutation induced by $\rev u$ maps the set $\{1,3,\ldots,2\lceil\frac{n}{2}\rceil-1\}$ to the set $\{\lfloor\frac{n}{2}\rfloor+1,\lfloor\frac{n}{2}\rfloor+2,\ldots,n\}$, and it maps the set
  $\{2,4,\ldots,2\lfloor\frac{n}{2}\rfloor\}$ to the set $\{1,2,\ldots,\lfloor\frac{n}{2}\rfloor\}$.
  The action of $\Join\left\{ a \in \Atoms_M : a \ne \left\lfloor\frac{n}{2}\right\rfloor\right\}$
  performs a
  half twist of the numbers greater than~$\frac{n}{2}$ and changes the
  sign of the numbers $1,2,\ldots,\lfloor\frac{n}{2}\rfloor$.
  Hence, for all $i>0$ we have
  \[ (2i\pm 1)\cdot x_3 > \frac{n}{2} > 0 > (2i)\cdot x_3 \] 
  and thus $\Start(x_3) = \{ 1, 3, \ldots 2\left\lfloor\frac{n}{2}\right\rfloor - 1\}$.
\end{proof}

\begin{remark}
Calculating powers of the adjacency matrices, it can be shown that the acceptor~$\Acceptor_{\Artin{B}_n}$ has a diameter of~5 for $n\in\{5,6,7,8\}$, so the statement of \autoref{type-B-ess-trans} cannot be sharpened in general.
The diameter of~$\Acceptor_{\Artin{B}_2}$ is~2, and the acceptors~$\Acceptor_{\Artin{B}_3}$ and~$\Acceptor_{\Artin{B}_4}$ have a diameter of~4.
\end{remark}

\begin{proposition} \label{type-D-ess-trans}
  Suppose that $M$ is the Artin monoid of type $\Artin{D}_n$, where $n\ge 3$.
  \[
    \begin{xy}
      0;<2em,0em>:<0em,1em>::
      (0,1)*+{0}="0";
      (0,-1)*+{1}="1";
      (1,0)*+{2}="2";
      (2,0)*+{3}="3";
      (3,0)*+{4}="4";
      (3.75,0)="5";
      (4.6,0)="6";
      (6,0)*+{(n-1)}="7";
      {\ar@{-}  "0";"2"};
      {\ar@{-}  "1";"2"};
      {\ar@{-}  "2";"3"};
      {\ar@{-}  "3";"4"};
      {\ar@{-} "4";"5"};
      {\ar@{..}  "5";"6"};
      {\ar@{-}  "6";"7"};
    \end{xy}
  \]
  Then $M$ is essentially 5-transitive.
\end{proposition}
\begin{proof}
  Let $x_2 = u(1,2, \ldots, n-1)$ and $x_3 = v(1,2,\ldots,n-1)$.
  As in the proof of \autoref{type-A-ess-trans}, it suffices to show that $\Start(x_3) \subseteq \{ 1, 3, \ldots 2\left\lfloor\frac{n}{2}\right\rfloor - 1\}$ by \autoref{Artin-normal-form}, \autoref{Artin-ProperSimplesEssential}, \autoref{connecting-atoms}, \autoref{one-to-many}, and \autoref{rev-u}.

  Simple elements in $\Artin{D}_n$ correspond to those signed permutations of $\{1,\ldots,n\}$ that
  change the sign of an even number of integers.  As in type
  $\Artin{B}_n$, the atoms $i = 1, \ldots, n-1$ correspond to the signed
  transpositions $(i,i+1)(-i,-(i+1))$, but now the atom~$0$ corresponds to the signed transposition
  $(1,-2)(2,-1)$.
  The starting set can be computed from the induced signed
  permutation \cite[Proposition 8.2.2]{CombCox}: One has
  \[
    \Start(x_3) = \{ i \in \Atoms : i\cdot x_3 > (i+1)\cdot x_3\}
    \;,
  \]
  where we use the convention $0\cdot x_3 = -2\cdot x_3$.

  The signed permutation induced by $\rev u$ maps the set $\{1,3,\ldots,2\lceil\frac{n}{2}\rceil-1\}$ to the set $\{\lfloor\frac{n}{2}\rfloor+1,\lfloor\frac{n}{2}\rfloor+2,\ldots,n\}$, and it maps the set
  $\{2,4,\ldots,2\lfloor\frac{n}{2}\rfloor\}$ to the set $\{1,2,\ldots,\lfloor\frac{n}{2}\rfloor\}$.
  The action of $\Join\left\{ a \in \Atoms_M : a \ne \left\lfloor\frac{n}{2}\right\rfloor\right\}$ performs a
  half twist of the numbers greater than~$\frac{n}{2}$, changes the
  sign of the integers $2,3,\ldots,\lfloor\frac{n}{2}\rfloor$, and the image of~$1$ is contained in $\{-1,1\}$.
  Hence, for all $i>0$ we have
  \[ (2i\pm 1)\cdot x_3 > \frac{n}{2} > 1 \ge (2i)\cdot x_3 \]
  as well as $0\cdot x_3 = -2\cdot x_3 \le \frac{n}{2},$
  whence $\Start(x_3) = \{ 1, 3, \ldots
  2\left\lfloor\frac{n}{2}\right\rfloor - 1\}$.
\end{proof}

\begin{remark}
Calculating powers of the adjacency matrices, it can be shown that the acceptor~$\Acceptor_{\Artin{D}_n}$ has a diameter of~5 for $n\in\{6,7,8\}$, so the statement of \autoref{type-D-ess-trans} cannot be sharpened in general.
The acceptors~$\Acceptor_{\Artin{D}_3}$, $\Acceptor_{\Artin{D}_4}$ and~$\Acceptor_{\Artin{D}_5}$ have a diameter of~4.

Note that the monoid~$\Artin{A}_4$ is a parabolic submonoid of the monoid~$\Artin{D}_5$, yet its acceptor has a larger diameter than that of the monoid~$\Artin{D}_5$.
\end{remark}

\begin{proposition} \label{type-E6-ess-trans}
  Suppose that $M$ is the Artin monoid of type $\Artin{E}_6$.
  \[
    \begin{xy}
      0;<2em,0em>:<0em,2em>::
      (0,0)*+{1}="1";
      (1,0)*+{2}="2";
      (2,0)*+{3}="3";
      (3,0)*+{4}="4";
      (4,0)*+{5}="5";
      (2,1)*+{0}="0";
      {\ar@{-}  "1";"2"};
      {\ar@{-}  "2";"3"};
      {\ar@{-}  "3";"4"};
      {\ar@{-}  "4";"5"};
      {\ar@{-}  "3";"0"};
    \end{xy}
  \]
  Then $M$ is essentially 4-transitive.
\end{proposition}
\begin{proof}
  \autoref{Artin-ProperSimplesEssential} yields $\Ess=\pSimples\neq\emptyset$.
  By \autoref{connecting-atoms} and \autoref{Artin-normal-form}, it suffices to construct an element
  $x_3 \in \pSimples$ with $\Start(x_3) = \{3\}$ and $\Fin(x_3) =
  \Atoms \setminus \{3\}$.

  By a direct computation in the Coxeter group of type $\Artin{E}_6$, one readily verifies that the element
  $x_3 = 302134302154$ is simple and has the required starting and finishing sets.
\end{proof}

\begin{remark}
Calculating powers of the adjacency matrix, it can be shown that the diameter of~$\Acceptor_{\Artin{E}_6}$ is~4, so the statement of \autoref{type-E6-ess-trans} cannot be sharpened.
\end{remark}

\begin{proposition} \label{type-E7-ess-trans}
  Suppose that $M$ is the Artin monoid of type $\Artin{E}_7$.
  \[
    \begin{xy}
      0;<2em,0em>:<0em,2em>::
      (0,0)*+{1}="1";
      (1,0)*+{2}="2";
      (2,0)*+{3}="3";
      (3,0)*+{4}="4";
      (4,0)*+{5}="5";
      (5,0)*+{6}="6";
      (2,1)*+{0}="0";
      {\ar@{-} "1";"2"};
      {\ar@{-} "2";"3"};
      {\ar@{-} "3";"4"};
      {\ar@{-} "4";"5"};
      {\ar@{-} "5";"6"};
      {\ar@{-} "3";"0"};
    \end{xy}
  \]
  Then $M$ is essentially 4-transitive.
\end{proposition}
\begin{proof}
  \autoref{Artin-ProperSimplesEssential} yields $\Ess=\pSimples\neq\emptyset$.
  By \autoref{connecting-atoms} and \autoref{Artin-normal-form}, it suffices to construct an element
  $x_3 \in \pSimples$ with $\Start(x_3) = \{3\}$ and $\Fin(x_3) =
  \Atoms \setminus \{3\}$.

  By a direct computation in the Coxeter group of type $\Artin{E}_7$, one readily verifies that the element
  $x_3 = 302134302134543021654$ is simple and has the required starting and finishing sets.
\end{proof}

\begin{remark}
Calculating powers of the adjacency matrix, it can be shown that the diameter of~$\Acceptor_{\Artin{E}_7}$ is~4, so the statement of \autoref{type-E7-ess-trans} cannot be sharpened.
\end{remark}

\begin{proposition} \label{type-E8-ess-trans}
  Suppose that $M$ is the Artin monoid of type $\Artin{E}_8$.
  \[
    \begin{xy}
      0;<2em,0em>:<0em,2em>::
      (0,0)*+{1}="1";
      (1,0)*+{2}="2";
      (2,0)*+{3}="3";
      (3,0)*+{4}="4";
      (4,0)*+{5}="5";
      (5,0)*+{6}="6";
      (6,0)*+{7}="7";
      (2,1)*+{0}="0";
      {\ar@{-} "1";"2"};
      {\ar@{-} "2";"3"};
      {\ar@{-} "3";"4"};
      {\ar@{-} "4";"5"};
      {\ar@{-} "5";"6"};
      {\ar@{-} "6";"7"};
      {\ar@{-} "3";"0"};
    \end{xy}
  \]
  Then $M$ is essentially 4-transitive.
\end{proposition}
\begin{proof}
  \autoref{Artin-ProperSimplesEssential} yields $\Ess=\pSimples\neq\emptyset$.
  By \autoref{connecting-atoms} and \autoref{Artin-normal-form}, it suffices to construct an element
  $x_3 \in \pSimples$ with $\Start(x_3) = \{3\}$ and $\Fin(x_3) =
  \Atoms \setminus \{3\}$.

  By a direct computation in the Coxeter group of type $\Artin{E}_8$, one readily verifies that the element
  $x_3 = 30213430213454302134565430217654$ is simple and has the required starting and finishing sets.
\end{proof}

\begin{remark}
Calculating powers of the adjacency matrix, it can be shown that the diameter of~$\Acceptor_{\Artin{E}_8}$ is~4, so the statement of \autoref{type-E8-ess-trans} cannot be sharpened.
\end{remark}

\begin{proposition} \label{type-H3-ess-trans}
  Suppose that $M$ is the Artin monoid of type $\Artin{H}_3$.
  \[
    \begin{xy}
      0;<2em,0em>:<0em,2em>::
      (0,0)*+{1}="1";
      (1,0)*+{2}="2";
      (2,0)*+{3}="3";
      {\ar@{-}^{5}  "1";"2"};
      {\ar@{-}     "2";"3"};
    \end{xy}
  \]
  Then $M$ is essentially 3-transitive.
\end{proposition}
\begin{proof}
  \autoref{Artin-ProperSimplesEssential} yields $\Ess=\pSimples\neq\emptyset$.
  By \autoref{connecting-atoms} and \autoref{Artin-normal-form}, it suffices to construct for every atom $a\in\Atoms$ an element
  $x_a \in \pSimples$ such that $\Start(x_a) = \{a\}$ and $\Fin(x_a) =
  \Atoms \setminus \{2\}$.

  By a direct computation in the Coxeter group of type $\Artin{H}_3$, one readily verifies that the elements
  $x_1 = 1213$, $x_2 = 213$ and $x_3 = 321213$ are simple and have the required starting and finishing sets.
\end{proof}

\begin{remark}
Calculating powers of the adjacency matrix, it can be shown that the diameter of $\Acceptor_{\Artin{H}_3}$ is~3, so the statement of \autoref{type-H3-ess-trans} cannot be sharpened.
\end{remark}

\begin{proposition} \label{type-H4-ess-trans}
  Suppose that $M$ is the Artin monoid of type $\Artin{H}_4$.
  \[
    \begin{xy}
      0;<2em,0em>:<0em,2em>::
      (0,0)*+{1}="1";
      (1,0)*+{2}="2";
      (2,0)*+{3}="3";
      (3,0)*+{4}="4";
      {\ar@{-}^{5}  "1";"2"};
      {\ar@{-}     "2";"3"};
      {\ar@{-}     "3";"4"};
    \end{xy}
  \]
  Then $M$ is essentially 3-transitive.
\end{proposition}
\begin{proof}
  \autoref{Artin-ProperSimplesEssential} yields $\Ess=\pSimples\neq\emptyset$.
  By \autoref{connecting-atoms} and \autoref{Artin-normal-form}, it suffices to construct for every atom $a\in\Atoms$ an element
  $x_a \in \pSimples$ such that $\Start(x_a) = \{a\}$ and $\Fin(x_a) =
  \Atoms \setminus \{2\}$.

  By a direct computation in the Coxeter group of type $\Artin{H}_4$, one readily verifies that the elements
  $x_1 = 121232143$, $x_2 = 21232143$, $x_3 = 32121343$ and $x_4 = 4321213212343212134$
  are simple and have the required starting and finishing sets.
\end{proof}

\begin{remark}
Calculating powers of the adjacency matrix, it can be shown that the diameter of $\Acceptor_{\Artin{H}_4}$ is~3, so the statement of \autoref{type-H4-ess-trans} cannot be sharpened.
\end{remark}

\begin{proposition} \label{type-F4-ess-trans}
  Suppose that $M$ is the Artin monoid of type $\Artin{F}_4$.
  \[
    \begin{xy}
      0;<2em,0em>:<0em,2em>::
      (0,0)*+{1}="1";
      (1,0)*+{2}="2";
      (2,0)*+{3}="3";
      (3,0)*+{4}="4";
      {\ar@{-}     "1";"2"};
      {\ar@{-}^{4}  "2";"3"};
      {\ar@{-}     "3";"4"};
    \end{xy}
  \]
  Then $M$ is essentially 3-transitive.
\end{proposition}
\begin{proof}
  By \autoref{Artin-ProperSimplesEssential}, one has $\Ess=\pSimples\neq\emptyset$.
  By \autoref{Artin-normal-form}, it suffices to construct for every atom $a\in\Atoms$ an element $x_a \in \pSimples$ such that $\Start(x_a) = \{a\}$ and $\Fin(x_a) = \{1,3\}$ as well as an element $x_{\bar{a}}\in\pSimples$ such that $\Start(x_{\bar{a}}) = \{1,3\}$ and $\Fin(x_{\bar{a}}) = \Atoms \setminus \{a\}$.

  By a direct computation in the Coxeter group of type $\Artin{F}_4$, one readily verifies that the elements
  \[\begin{array}{l@{\hspace{15pt}}l@{\hspace{15pt}}l@{\hspace{15pt}}l}
    x_1 = 1232143 &
    x_2 = 213 &
    x_3 = 3213 &
    x_4 = 43213 \\[1ex]
    x_{\bar{1}} = 132132343234 &
    x_{\bar{2}} = 1321343 &
    x_{\bar{3}} = 13214 &
    x_{\bar{4}} = 12321324321323
  \end{array}\]
  are simple and have the required starting and finishing sets.
\end{proof}

\begin{remark}
Calculating powers of the adjacency matrix, it can be shown that the diameter of $\Acceptor_{\Artin{F}_4}$ is~3, so the statement of \autoref{type-F4-ess-trans} cannot be sharpened.

We note that, for $a\in\{1,4\}$, there exists no simple element~$x$ satisfying $\Start(x) = \{a\}$ and $|\Fin(x)| = 3$, and there exists no simple element~$x$ satisfying $|\Start(x)| = 1$ and $\Fin(x) = \Atoms\setminus\{a\}$.
Hence, it is not possible to argue as in the proofs of \autoref{type-H3-ess-trans} and \autoref{type-H4-ess-trans}.
\end{remark}

\begin{proposition} \label{type-I2-ess-trans}
  Suppose that $M$ is the Artin monoid of type $\Artin{I}_2(p)$, with $p \ge 3$.
  \[
    \begin{xy}
      0;<2em,0em>:<0em,2em>::
      (0,0)*+{1}="1";
      (1,0)*+{2}="2";
      {\ar@{-}^{p}  "1";"2"};
    \end{xy}
  \]
  Then $M$ is essentially 2-transitive.
\end{proposition}
\begin{proof}
  We have $p\ge3$ by assumption, so $12$ and $21$ are distinct proper simple elements.  The proper simple elements fall into one of four types:
  \[\begin{array}{@{}l@{\hspace{1.7ex}}ll@{\hspace{1cm}}l@{\hspace{1.7ex}}ll@{}}
   \textup{(a)} & 2(12)^k=(21)^k2 & (2k+1<p)
     &
   \textup{(b)} & 1(21)^k2 & (2k+2<p)
     \\[1ex]
   \textup{(c)} & 2(12)^k1 & (2k+2<p)
     &
   \textup{(d)} & 1(21)^k =(12)^k1 & (2k+1<p)
  \end{array}\]

  \noindent For $t\in\{\textup{a},\textup{b},\textup{c},\textup{d}\}$, let $t_1$ and $t_2$ denote two arbitrary simple elements of type~$(t)$.  We then have\vspace{-1ex}
  \[\begin{array}{l@{\qquad}l@{\qquad}l@{\qquad}l}
  \textup{a}_1|\textup{a}_2
      & \textup{a}_1|21|\textup{b}_2
      & \textup{a}_1|\textup{c}_2
      & \textup{a}_1|21|\textup{d}_2 \\[1ex]
  \textup{b}_1|\textup{a}_2
      & \textup{b}_1|21|\textup{b}_2
      & \textup{b}_1|\textup{c}_2
      & \textup{b}_1|21|\textup{d}_2 \\[1ex]
  \textup{c}_1|12|\textup{a}_2
      & \textup{c}_1|\textup{b}_2
      & \textup{c}_1|12|\textup{c}_2
      & \textup{c}_1|\textup{d}_2    \\[1ex]
  \textup{d}_1|12|\textup{a}_2
      & \textup{d}_1|\textup{b}_2
      & \textup{d}_1|12|\textup{c}_2
      & \textup{d}_1|\textup{d}_2
  \end{array}\]

  \noindent and thus $M$ is $2$-transitive.
\end{proof}

Combining Propositions~\ref{type-A-ess-trans} to \ref{type-I2-ess-trans} with the classification of Artin monoids of spherical type from~\autoref{T:ArtinClassification} and~\autoref{reducible-non-transitive}, we have the following theorem:

\begin{theorem}\label{Artin-EssentiallyTransitive}
  Let $M$ be an Artin monoid of spherical type with more than one atom.

  The language of normal forms in~$M$ is essentially transitive if and only if~$M$ is irreducible.  Moreover,
  if the language is essentially transitive then it is essentially $5$-transitive.
\end{theorem}

\begin{corollary}\label{C:BoundedExpectedPD-Artin}
Let $M$ be an irreducible Artin monoid of spherical type, let $\nu_{k}$ be the uniform probability measure on $\calL_M^{(k)}$, and let $\mu_\Atoms$ be the uniform probability distribution on the set $\Atoms$ of atoms of $M$.

The expected value $\Expect_{\nu_{k} \times \mu_\Atoms}[\pd]$ of the penetration distance with respect to $\nu_k\times \mu_\Atoms$ is uniformly bounded (that is, bounded independently of $k$).
\end{corollary}
\begin{proof}
\autoref{C:BoundedExpectedPD}, \autoref{Artin-ProperSimplesEssential} and \autoref{Artin-EssentiallyTransitive} imply the claim.
\end{proof}

Recall that $\beta_M$ is the exponential growth rate of the regular language~$\calL_M^{(k)}$.

\begin{lemma}\label{Artin-Growth}
If $M$ is an irreducible Artin monoid of spherical type with more than one atom, then one has $\beta_M>1$.
\end{lemma}
\begin{proof}
Consider two atoms~$a\neq b$ of~$M$.  As~$\calL_M$ is essentially transitive, there exist $s_1,\ldots,s_k,t_1,\ldots,t_\ell\in\pSimples$ such that $a|s_1|\cdots|s_k|b|t_1|\cdots t_\ell|a$.  Moreover, we have $a|a|\cdots|a$ by \autoref{Artin-normal-form}.
Thus, one has $\calL_M^{\left(1+N(k+\ell+2)\right)} \ge 2^N$, showing the claim.
\end{proof}

\begin{corollary}\label{C:UnboundedExpectedPD-ZappaSzep-Artin}
Let $M=G\zs H$, where $G$ and $H$ are irreducible Artin monoids of spherical type with more than one atom, let $\nu_{k}$ be the uniform probability measure on~$\calL_M^{(k)}$, and let~$\mu_\Atoms$ be the uniform probability distribution on the set~$\Atoms$ of atoms of~$M$.

The expected value
$\Expect_{\nu_{k} \times \mu_\Atoms}[\pd]$ diverges, that is, $\lim_{k\to\infty} \Expect_{\nu_{k} \times \mu_\Atoms}[\pd] = \infty$.
\end{corollary}
\begin{proof}
The claim follows from \autoref{T:PSeqProduct}, \autoref{Artin-ProperSimplesEssential}, \autoref{Artin-EssentiallyTransitive} and \autoref{Artin-Growth}.
\end{proof}
\medskip

In the terminology of this paper, Dehornoy asked in~\cite[Question~3.13]{DehornoyJCTA07} whether for the braid monoid, that is in the situation $M=\Artin{A}_n$, one has
$\big|\calLbar^{(k)}(s)\big|\in\Theta\big(\big|\calLbar^{(k)}\big|\big)$ for all~$s\in\pSimples$, where
\[
  \calLbar^{(k)}(s) := \calLbar^{(k)} \cap \Simples^*\ldot s
    = \Big\{ s_1\ldot\cdots\ldot s_k\in \calLbar^{(k)} : s_k = s \Big\}
  \;.
\]
The answer is affirmative for all irreducible Artin monoids of spherical type:

\begin{corollary}\label{C:Dehornoy1}
If $M$ is an irreducible Artin monoid of spherical type with more than one atom and $s\in\pSimples$, then one has
$\big|\calL^{(k)}(s)\big|\in\Theta\big(\big|\calL^{(k)}\big|\big)$.
\end{corollary}
\begin{proof}
The claim follows from \autoref{L:RestrictedNF}, \autoref{Artin-ProperSimplesEssential} and
\autoref{Artin-EssentiallyTransitive}.
\end{proof}

\begin{corollary}\label{C:Dehornoy2}
If $M$ is an irreducible Artin monoid of spherical type and $s\in\pSimples$, then one has
$\big|\calLbar^{(k)}(s)\big|\in\Theta\big(\big|\calLbar^{(k)}\big|\big)$.
\end{corollary}
\begin{proof}
The claim is vacuously true if $M$ has only one atom, so we may assume otherwise.
In this case, we have $\big|\calL^{(k)}(s)\big|\in\Theta\big(\big|\calL^{(k)}\big|\big)$ by \autoref{C:Dehornoy1} and $\beta>1$ by \autoref{Artin-Growth}.
The claim then follows from $\calLbar^{(k)}(s) = \bigsqcup_{j=1}^k\Delta^{k-j}\calL^{(j)}(s)$ and
$\calLbar^{(k)} = \bigsqcup_{j=1}^k\Delta^{k-j}\calL^{(j)}$ and \autoref{L:Sums}.
\end{proof}

\bibliographystyle{alpha-sjt}
\bibliography{bibliography}

\bigskip
\noindent
\begin{minipage}[t]{0.52\textwidth}
\noindent\textbf{Volker Gebhardt}\\
\noindent \texttt{v.gebhardt@westernsydney.edu.au}
\end{minipage}
\hfill
\begin{minipage}[t]{0.47\textwidth}
\noindent\textbf{Stephen Tawn}\\
\noindent \texttt{stephen@tawn.co.uk}\\
\noindent \url{http://www.stephentawn.info}
\end{minipage}
\medskip
\begin{center}
Western Sydney University \\
Centre for Research in Mathematics\\
Locked Bag 1797, Penrith NSW 2751, Australia\\
\noindent \url{http://www.westernsydney.edu.au/crm}
\end{center}

\end{document}